\documentclass{amsart}

\usepackage{amssymb}
\usepackage{cite}
\usepackage[dvipsnames]{xcolor}
\usepackage{graphics}
\usepackage{pinlabel}
\usepackage{bm}
\usepackage{transparent}

\theoremstyle{plain}
\newtheorem{theorem}{Theorem}[section]
\newtheorem{proposition}[theorem]{Proposition}
\newtheorem{corollary}[theorem]{Corollary}
\newtheorem{lemma}[theorem]{Lemma}
\newtheorem{claim}[theorem]{Claim}

\newtheorem*{question}{Question}
\newtheorem*{conjecture}{Conjecture}

\newtheorem*{anosov_closing_lemma}{Anosov Closing Lemma}
\newtheorem*{homotopy_closing_lemma}{Homotopy Closing Lemma}
\newtheorem*{recurrent_links_lemma}{Recurrent Links Lemma}
\newtheorem*{closed_orbits_theorem}{Closed Orbits Theorem}

\newtheorem*{morse_lemma}{Morse Lemma}
\newtheorem*{moores_theorem}{Moore's Theorem}

\theoremstyle{definition}
\newtheorem{definition}[theorem]{Definition}

\theoremstyle{remark}
\newtheorem{remark}[theorem]{Remark}
\newtheorem{example}[theorem]{Example}

\newcommand{\bR}{\mathbb{R}}
\newcommand{\bZ}{\mathbb{Z}}
\newcommand{\bH}{\mathbb{H}}

\newcommand{\acts}{\curvearrowright}
\newcommand{\racts}{\curvearrowleft}

\newcommand{\rquotient}[2]{{\left.\raisebox{.2em}{$#1$}\middle/\raisebox{-.2em}{$#2$}\right.}}

\newcommand{\cD}{\mathcal{D}}
\newcommand{\hcD}{\bm{\mathcal{D}}}

\newcommand{\cI}{\mathcal{I}}

\newcommand{\he}{\mathbf{e}}

\newcommand{\hP}{\mathbf{P}}

\newcommand{\flow}[1]{\langle #1 \rangle}

\newcommand{\tM}{\widetilde{M}}
\newcommand{\tPhi}{\widetilde{\Phi}}
\newcommand{\tTheta}{\widetilde{\Theta}}

\newcommand{\hK}{\mathbf{K}}
\newcommand{\hL}{\mathbf{L}}
\newcommand{\hN}{\mathbf{N}}

\newcommand{\chK}{\mathring{\mathbf{K}}}
\newcommand{\chL}{\mathring{\mathbf{L}}}
\newcommand{\chN}{\mathring{\mathbf{N}}}

\newcommand{\hA}{\mathbf{A}}
\newcommand{\hB}{\mathbf{B}}
\newcommand{\hC}{\mathbf{C}}

\newcommand{\hatm}{\widehat{m}}
\newcommand{\cM}{\mathcal{M}}

\newcommand{\chcD}{\mathring{\hcD}}

\definecolor{maroon}{RGB}{128,0,0}
\definecolor{navy}{RGB}{0,0,128}
\definecolor{dkgreen}{RGB}{0,128,0}

\DeclareMathOperator{\T}{T}
\DeclareMathOperator{\U}{U}
\DeclareMathOperator{\lk}{lk}

\title{Coarse hyperbolicity and closed orbits for quasigeodesic flows}
\author{Steven Frankel}
\address{Department of Mathematics\\Washington University in St Louis\\St Louis, MO}
\email{steven.frankel@wustl.edu}
\date{\today}

\begin{document}

\begin{abstract}
We prove a conjecture of Calegari's, that every quasigeodesic flow on a closed hyperbolic $3$-manifold contains a closed orbit.
\end{abstract}

\maketitle

\section{Introduction}
In 1950, Seifert conjectured that every nonsingular flow on the 3-sphere must contain a closed orbit \cite{Seifert}. The first counterexamples appeared in 1974, when Schweitzer showed that every nonsingular flow on a $3$-manifold can be deformed to a $C^1$ flow with no closed orbits \cite{Schweizer}. These examples have been generalized considerably, and it is known that the flow can be taken to be smooth \cite{KKuperberg} or volume-preserving \cite{GKuperberg}.

On the other hand, there are certain geometric constraints that ensure the existence of closed orbits. Taubes' 2007 proof of the 3-dimensional Weinstein conjecture shows that every Reeb flow on a closed 3-manifold has a closed orbit  \cite{Taubes}. Reeb flows are \emph{geodesible}, i.e. there is a Riemannian metric in which the flowlines are geodesics. In 2010, Rechtman showed that any real analytic geodesible flow on a closed 3-manifold has a closed orbit, unless the manifold is a torus bundle with reducible monodromy \cite{Rechtman}.

Geodesibility is a global geometric condition. In contrast, a flow is \emph{quasigeodesic} if the path taken by each point is coarsely comparable to a geodesic, a local condition. In this paper, we show that every quasigeodesic flow on a closed hyperbolic $3$-manifold contains a closed orbit, as conjectured by Calegari.

Our argument works by studying the transverse behavior of a quasigeodesic flow, which displays a kind of ``coarse hyperbolicity'' whenever the ambient manifold is hyperbolic. In fact, our argument can be applied to a larger class of flows with coarsely hyperbolic transverse structures.

In addition, one can think of this coarse transverse hyperbolicity as a topological analogue of the transverse hyperbolicity exhibited by a pseudo-Anosov flow. This analogy is taken further in a proof, announced by the author, of Calegari's Flow Conjecture, which asserts that any quasigeodesic flow on a closed hyperbolic $3$-manifold can be deformed to a flow that is both quasigeodesic and pseudo-Anosov.

\subsection{Quasigeodesic flows}
A \emph{flow} on a manifold $M$ is a continuous map 
\[ \Phi_{(\cdot)}(\cdot): \bR \times M \to M \]
such that 
\begin{itemize}
	\item $\Phi_0(x) = x$ for all $x \in M$, and 
	\item $\Phi_s(\Phi_t(x)) = \Phi_{t + s}(x)$ for all $x \in M$ and $t, s \in \bR$.
\end{itemize}
For each $t \in \bR$, the \emph{time-$t$ map} $\Phi_t(\cdot): M \to M$ is a homeomorphism, so we can think of a flow as a continuous action $M \racts \bR$, writing $x \cdot t = \Phi_t(x)$. A flow is \emph{nonsingular} if it has no global fixed points, i.e. for each $x \in M$ there is a time $t \in \bR$ such that $\Phi_t(x) \neq x$. 

A flow $\Phi$ on a manifold $M$ is \emph{quasigeodesic} if each orbit lifts to a quasigeodesic in the universal cover $\tM$. That is, each lifted orbit admits a parametrization $\gamma: \bR \to \tM$ satisfying
\[ \frac{1}{k} \cdot d(\gamma(s), \gamma(t)) - \epsilon \leq |s - t| \leq k \cdot d(\gamma(x), \gamma(y)) + \epsilon \]
for constants $k > 0, \epsilon \geq 0$ that may depend on the orbit. It is \emph{uniformly quasigeodesic} if the constants $k, \epsilon$ can be chosen independent of the orbit.

If $M$ is a closed hyperbolic $3$-manifold, then the universal cover $\tM$ is identified with $\bH^3$, so it has a natural compactification to a closed $3$-ball $\tM \sqcup S^2_\infty$. Here, the \emph{sphere at infinity} $S^2_\infty$ is identified with the boundary of hyperbolic space in the unit ball model. The action of $\pi_1(M)$ on the universal cover is isometric, so it extends to the sphere at infinity.

A flow $\Phi$ on $M$ lifts to a flow $\tPhi$ on $\tM$, but the orbits of the lifted flow need not behave well with respect to the sphere at infinity. In particular, they may remain in bounded subsets of $\tM$, or accumulate on arbitrary closed subsets of $S^2_\infty$. When $\Phi$ is quasigeodesic, however, the following so-called \emph{Morse Lemma} implies that each lifted orbit has well-defined and distinct endpoints in $S^2_\infty$. See \cite{Gromov}, \cite[Corollary~3.44]{Kapovich}, or \cite[\S III.H]{BridsonHaefliger}.

\begin{morse_lemma}\label{lemma:Morse}
	Every quasigeodesic in $\bH^3$ lies at a bounded distance from a unique geodesic. Furtherore, there are constants $C(k, \epsilon)$ such that every $(k, \epsilon)$-quasigeodesic in $\bH^3$ lies in the $C(k, \epsilon)$-neighborhood of its associated geodesic.
\end{morse_lemma}

In addition, the endpoints of lifted orbits vary continuously, and this behavior characterizes the quasigeodesic flows on a closed hyperbolic $3$-manifold. They are exactly the flows that can be studied ``from infinity'' in the following sense.

\begin{proposition}[\hspace{1sp}{\cite[Theorem~B]{FenleyMosher} \& \cite[Lemma~4.3]{Calegari}}]\label{proposition:QGFlowCharacterization}
	Let $\Phi$ be a flow on a closed hyperbolic $3$-manifold $M$, and let $\tPhi$ be the lifted flow on the universal cover $\tM$. Then $\Phi$ is quasigeodesic if and only if
	\begin{enumerate}
		\item each orbit of $\tPhi$ has well-defined and distinct endpoints in $S^2_\infty$, and
		\item the positive and negative endpoints of $x \cdot \bR$ vary continuously with $x \in \tM$.
	\end{enumerate}
\end{proposition}

The simplest examples of quasigeodesic flows come from fibrations.

\begin{example}
	Zeghib showed that any flow on a closed $3$-manifold $M$ (not necessarily hyperbolic) that is transverse to a fibration is quasigeodesic \cite{Zeghib}. The idea is to lift such a flow to the infinite cyclic cover dual to a fiber $\Sigma \subset M$, which may be identified with $\Sigma \times \bR$ in such a way that the lifts of $\Sigma$ are of the form $\Sigma \times \{i\}$ for $i \in \bZ$. Quasigeodesity follows from the observation that there are upper and lower bounds on the distance between adjacent lifts, as well as the time it takes for the flow to move points from one lift to the next.	
\end{example}

On the other hand, there are many quasigeodesic flows that are not transverse to fibrations, even virtually (i.e. after passing to a finite cover).

\begin{example}
	Gabai showed that any nontrivial second cohomology class on a closed $3$-manifold represents the depth-zero leaf of a taut, finite-depth foliation \cite{Gabai}. Fenley and Mosher showed that there is a quasigeodesic flow that is transverse or ``almost transverse'' to such a foliation \cite{FenleyMosher}. If one takes a cohomology class that is not virtually represented by a fiber, then the associated quasigeodesic flow is not virtually transverse to a fibration.
\end{example}

\subsection{Transverse hyperbolicity}
For motivation, we recall the Anosov Closing Lemma, which produces closed orbits for Anosov and pseudo-Anosov flows.

A smooth flow $\Phi$ on a $3$-manifold $M$ is \emph{Anosov} if it preserves a splitting of the tangent bundle $\T M$ into three one-dimensional sub-bundles
\[ \T M = E^s \oplus \T \Phi \oplus E^u \]
where $T \Phi$ is tangent to the flow, and the flow exponentially contracts the \emph{stable bundle} $E^s$ and exponentially expands the \emph{unstable bundle} $E^u$. Here and throughout, we use the convention that ``expansion'' means contraction in backwards time. This gives rise to a transverse pair of two-dimensional foliations, the \emph{weak stable} and \emph{weak unstable foliations}, obtained by integrating the $2$-dimensional sub-bundles $\T \Phi \oplus E^s$ and $\T \Phi \oplus E^u$.

A flow on a $3$-manifold is \emph{pseudo-Anosov} if it is Anosov everywhere except near a collection of isolated closed orbits, and the weak stable and unstable foliations on the complement of these orbits extend to \emph{singular} foliations on the entire manifold. Figure~\ref{figure:2dFoliations} illustrates the local picture near a singularity of order $4$.

\begin{figure}[h]
	\includegraphics{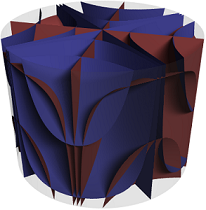}
	\caption{Singularities of stable and unstable foliations.} \label{figure:2dFoliations}
\end{figure}

Let $\phi: \Sigma \to \Sigma$ be a homeomorphism of a closed surface, and consider the surface bundle $M_\phi$ with monodromy $\phi$, i.e. the closed $3$-manifold
\[ M_\phi = \rquotient{\Sigma \times [0, 1]}{(x, 1) \sim (\phi(x), 0)}. \]
The semi-flow $(x, y) \cdot t = (x, y + t)$ on $M \times [0, 1]$ glues up to produce a flow $\Phi$ on $M_\phi$ called the \emph{suspension flow} associated to $\phi$.

The suspension flow associated to an arbitrary homeomorphism is quasigeodesic since it is transverse to the fibration of $M_\phi$ by the images of the surfaces $\Sigma \times \{y\}$. In addition, the suspension flow associated to a pseudo-Anosov homeomorphism is a pseudo-Anosov flow. Its $2$-dimensional (singular) weak stable and unstable foliations can be seen explicitly by flowing the $1$-dimensional (singular) stable and unstable foliations of the associated homeomorphism, thought of as living on a fiber.

An \emph{almost-cycle} in a flow is a closed loop obtained by concatenating a long flow segment with a short arc. More concretely, an \emph{$(\epsilon, T)$-cycle} consists of flow segment of the form $x \cdot [0, t]$ for $t \geq T$, concatenated with arc from $x$ to $x \cdot t$ with length at most $\epsilon$. The Anosov Closing Lemma leverages the contracting/expanding, or ``transversely hyperbolic'' behavior of an Anosov flow to find closed orbits near almost-cycles.

\begin{anosov_closing_lemma}
	Let $\Phi$ be an Anosov flow on a closed $3$-manifold $M$. Then for each $D > 0$ there are constants $T, \epsilon > 0$ such that any $(\epsilon, T)$-cycle contains a closed orbit in its $D$-neighborhood.
\end{anosov_closing_lemma}

An analogous result holds for pseudo-Anosov flows, but one must be careful with almost-cycles that lie near the singular orbits \cite{Mangum}.

The idea behind the Anosov Closing Lemma is illustrated in Figure~\ref{figure:AnosovClosing}. The left side of the figure depicts the local structure near the ends of a long flow segment $x \cdot [0, t]$, while the right side depicts the local structure near a point $x' = x \cdot \frac{t}{2}$ in the middle. Since $x$ is close to $x \cdot t$, the local stable/unstable leaf through $x$ intersects the local unstable/stable leaf through $x \cdot t$.

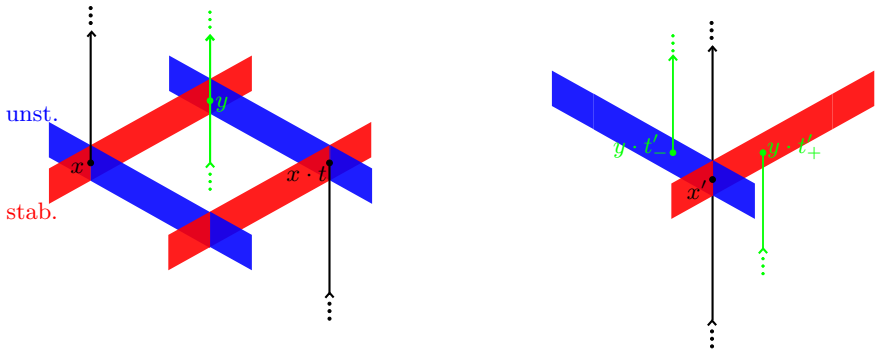
\begin{figure}[ht]
	\labellist
	\small\hair 2pt
	
	\pinlabel $x$ [t] at 31 83
	\pinlabel ${x \cdot t}$ [r] at 127 77
	\pinlabel $\textcolor{green}{y}$ [l] at 81 103
	\pinlabel \textcolor{red}{stab.} [t] at 14 68
	\pinlabel \textcolor{blue}{unst.} [b] at 14 95
	
	\pinlabel $x'$ [tr] at 270 76
	\pinlabel $\textcolor{green}{y \cdot t_-'}$ [r] at 256 87
	\pinlabel $\textcolor{green}{y \cdot t_+'}$ [l] at 289 87
	
	\endlabellist
	\centering
	\includegraphics{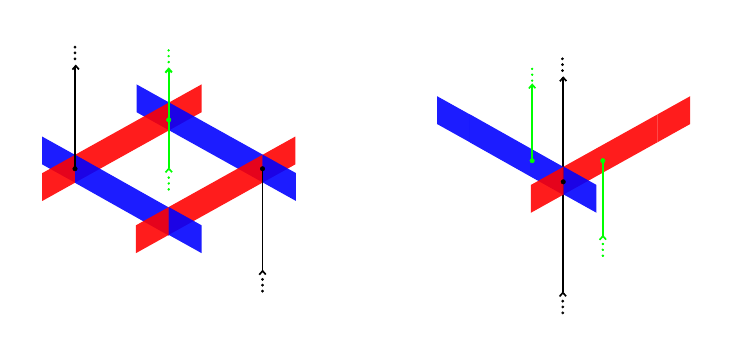}
	\caption{The Anosov closing lemma.} \label{figure:AnosovClosing}
\end{figure}

Take a point $y$ where the stable leaf through $x$ intersects the unstable leaf through $x \cdot t$. Flowing forward, we arrive at a point $y \cdot t_+'$ which lies very close to $x'$ along its stable leaf; flowing backwards, we arrive at a point $y \cdot t_-'$ which lies very close to $x'$ along its unstable leaf. This produces a flow segment $y \cdot [t_-', t_+']$ whose length is comparable to $x \cdot [0, t]$, but whose ends are much closer together, which we can close up to obtain an almost-cycle. Repeating, one obtains a sequence of better and better almost-cycles, which limit to a closed orbit.

\subsection{Coarse transverse hyperbolicity}
At first glance, the tangent condition that defines a quasigeodesic flow seems unrelated to the hyperbolic transverse structure that defines a pseudo-Anosov flow. In the presence of ambient hyperbolicity, however, we will see that a quasigeodesic flows admits an analogous transverse structure that is \emph{coarsely} hyperbolic.

In a pseudo-Anosov flow, a weak stable or unstable leaf is a maximal connected set of orbits that are forwards or backwards asymptotic to each other. In a quasigeodesic flow, one defines a \emph{weak positive} or \emph{weak negative leaf} to be a maximal connected set of orbits that share their positive or negative endpoints in the universal cover. Any pair of orbits in a weak positive or negative leaf are \emph{coarsely} forwards or backwards asymptotic, in the sense that they contain forwards or backwards rays that lie a uniformly bounded distance apart in the universal cover. 

This ``coarse transverse hyperbolicity'' is a far cry from the strict form of hyperbolicity needed for the Anosov closing lemma. Besides the coarseness of the contraction/expansion, there is little control over the topology of the weak leaves. They may have nontrivial interior, and may not be path-connected or even locally connected. In particular, there is no natural notion of transversality between weak positive and negative leaves. Nevertheless, we will prove a ``Homotopy Closing Lemma,'' which finds closed orbits in the free homotopy classes of certain almost-cycles. This can be seen as a coarse analogue of the Anosov Closing Lemma, which finds closed orbits geometrically near certain almost-cycles.

\subsection{Outline and results}
Fix a quasigeodesic flow $\Phi$ on a closed hyperbolic $3$-manifold $M$, which lifts to a flow $\tPhi$ on the universal cover $\tM \simeq \bH^3$. To avoid confusion, the orbits of $\tPhi$ will be called \emph{flowlines}. In $\tM$, the collections of all weak positive and weak negative leaves form a pair of $\pi_1(M)$-equivariant decompositions, $\cD^+_w$ and $\cD^-_w$.

Because of the topological and geometric difficulties in working with weak leaves, our argument will take a very different form than that of the Anosov Closing Lemma. Instead of working directly in the manifold, we will reduce the $3$-dimensional problem of finding closed orbits to a $2$-dimensional problem in the \emph{flowspace}. This is a topological plane $P$, the orbit space of the lifted flow, which comes equipped with a natural action of $\pi_1(M)$. See \S\ref{section:Leaves}.

Each point in $P$ corresponds to a flowline in $\tM$, which projects to an orbit in $M$, and the periodicity or recurrence of this orbit can be seen in terms of the flowspace action. In particular, a point in $P$ corresponds to a closed orbit in $M$ if and only if it is fixed by some nontrivial element of $\pi_1(M)$. Thus, to show the existence of a closed orbit it suffices to show that the flowspace action is not free.

The flowspace is also useful for understanding the topology of the weak leaves. The decompositions $\cD^\pm_w$ of $\tM$ project to decompositions $\cD^\pm$ of $P$, whose elements are simply called \emph{positive leaves} and \emph{negative leaves}. These \emph{positive and negative decompositions} are monotone, unbounded, and intersect compactly. That is, each positive or negative leaf is a closed, connected, unbounded planar set, and the intersection between a positive leaf and a negative leaf is compact. Note that the unboundedness of the leaves represents a ``transverse unboundedness'' of the weak leaves, which would be difficult to express without using the flowspace. 

We will use these properties to treat the positive and negative decompositions as a generalization of a pair of singular foliations. This analogy is realized by ignoring the local topology of leaves and focusing on their separation properties, which can be understood using Calegari's \emph{universal circle} \cite{Calegari}, a topological circle $S^1_u$, equipped with a faithful action of $\pi_1(M)$. Although the original construction of this circle is abstract, coming from a cyclic ordering of the Freudenthal ends of leaves, we showed in \cite{Frankel_qgflows} that it can be thought of as the boundary of the flowspace. That is, the disjoint union $\hP = P \sqcup S^1_u$ has a natural topology with which it is homeomorphic to a closed disc, and the flowspace and universal circle actions combine to form an action $\pi_1(M) \acts \hP$.

In \S\ref{section:Sprigs}, we will see that decompositions $\cD^\pm$ of $P$ extend to decompositions $\hcD^\pm$ of $\hP$ whose elements are called \emph{positive sprigs} and \emph{negative sprigs}. Each sprig has a natural set of \emph{ends}, the points at which it intersects the universal circle, and the separation properties of a sprig are reflected in those of its ends. Although sprigs contain leaves, and hence display the same kinds of topological pathologies, they are still easier to deal with since they are compact and have nice convergence properties.

In \S\ref{section:MasterSprigs}, we study the relationship between the positive and negative sprig decompositions, and show that they form a \emph{spidery pair} (cf. Definitions~\ref{definition:SpideryDecomposition} and \ref{definition:SpideryPair}), whose topological properties are studied in \S\ref{section:Separation}--\ref{section:Linking}. Although there is no natural notion of a transverse intersection point between a positive and negative sprig, there is a generalization of this idea called a \emph{linked point}. A point $p \in P$ is said to be linked if the ends of its positive and negative sprigs contain $0$-spheres that are linked in the $1$-sphere $S^1_u$. We show that the set $P_l \subset P$ of linked points has the following properties.

\begin{recurrent_links_lemma}
	The linked region $P_l$ is closed, nontrivial, $\pi_1(M)$-invariant, and contains an $\omega$-recurrent point.
\end{recurrent_links_lemma}

Here, an \emph{$\omega$-recurrent point} is one that corresponds to an $\omega$-recurrent orbit in $M$. If $p \in P$ is $\omega$-recurrent, then a sequence of almost-cycles $(\gamma_i)$ that approximate the corresponding forward orbit are represented by a sequence of elements $(g_i)$ in the fundamental group, called an \emph{$\omega$-sequence}, with the property that $\lim g_i(p) = p$.

In \S\ref{section:CoarseHyperbolicity}, we show that the coarse contraction of positive sprigs is reflected in the action of an $\omega$-sequence on the flowspace. In particular, if $p \in P$ is an $\omega$-recurrent point, then a corresponding $\omega$-sequence $(g_i)$ takes any point $p' \in P$ in the same positive sprig towards an \emph{a priori} determined compact region. We use this in \S\ref{section:ClosedOrbits} to prove our closing lemma.
 
\begin{homotopy_closing_lemma}
	Let $p \in P_l$ be an $\omega$-recurrent point, and let $(g_i)_{i = 1}^\infty$ be a corresponding $\omega$-sequence. Then for all $i$ sufficiently large, $g_i$ fixes a point in $P$.
\end{homotopy_closing_lemma}

Since a point in $P$ that is fixed by a nontrivial element of $\pi_1(M)$ corresponds to a closed orbit, our main theorem follows immediately.

\begin{closed_orbits_theorem}
	Every quasigeodesic flow on a closed hyperbolic $3$-manifold contains a closed orbit.
\end{closed_orbits_theorem}

In addition, we will see in \S\ref{subsection:ClosedOrbitsUC} that certain closed orbits can be seen purely in terms of the universal circle.

\subsection{Acknowledgements}
I would like to thank the referee, whose comments and suggestions contributed greatly to the readability of this paper. Thanks to Thierry Barbot, Thomas Barthelme, Danny Calegari, Sergio Fenley, Dave Gabai, Yair Minsky, Lee Mosher, and Juliana Xavier for many productive conversations. 

This material is based partially upon work supported by the National Science Foundation under Grant No. DMS-1611768. Any opinions, findings, and conclusions or recommendations expressed in this material are those of the author and do not necessarily reflect the views of the National Science Foundation.

\section{Topological background}
In this section we will review some topological background that will be used throughout the sequel. See \cite{Wilder}, \cite{Kuratowski1}, \cite{Kuratowski2}, and \cite{HockingYoung} for more details.

%
%
%
%

\subsection{Limits}
Let $(A_i)_{i=1}^\infty$ be a sequence of subsets of a metric space $X$. The \emph{limit inferior} 
\[ \varliminf A_i \]
is the set of all $x \in X$ such that each neighborhood of $x$ intersects all but finitely many of the $A_i$. The \emph{limit superior} 
\[ \varlimsup A_i \]
is the set of all $x \in X$ such that every neighborhood of $x$ intersects infinitely many of the $A_i$.

In other words, $x \in \varliminf A_i$ if and only if there is a sequence of points $x_i \in A_i$ that converge to $x$, and $y \in \varlimsup A_i$ if and only if there is a sequence of points $y_i \in A_i$ that accumulate on $y$. The limits inferior and superior are always closed, and 
\[ \varliminf A_i \subset \varlimsup A_i. \]
If the limits inferior and superior of agree, then $(A_i)$ is said to be \emph{(Kuratowski) convergent}, and we write
\[ \lim A_i := \varliminf A_i = \varlimsup A_i. \]

When $X$ is a compact metric space, this is equivalent to Hausdorff convergence, and we can avail ourselves of the following useful properties of Hausdorff limits \cite[\S 2]{HockingYoung}.

\begin{theorem}
	Let $X$ be compact metrizable space $X$. Then:
	\begin{enumerate}
		\item[(1)] Every sequence of subsets has a convergent subsequence.
		\item[(2)] If $(A_i)$ is a sequence of connected subsets, and $\varliminf A_i$ is not empty, then $\varlimsup A_i$ is connected.
		\item[(2')] If $(A_i)$ is a convergent sequence of connected subsets then $\lim A_i$ is connected.
	\end{enumerate}
\end{theorem}

\subsection{Decompositions}\label{subsection:Decompositions}
Let $X$ be a topological space. A \emph{partition} of $X$ is a collection of pairwise disjoint subsets that cover $X$. A \emph{decomposition} of $X$ is a partition whose elements are closed. A partition or decomposition is \emph{nontrivial} if it contains more than one element, and \emph{monotone} if its elements are connected.

Let $\cD$ be a decomposition of a space $X$. A subset $A \subset X$ is said to be \emph{$\cD$-saturated} if each decomposition element $K \in \cD$ that intersects $A$ is contained in $A$; equivalently, if $A$ is a union of decomposition elements. The \emph{$\cD$-saturation} of a subset $B \subset X$ is the smallest $\cD$-saturated set that contains $B$, which we denote by $\cD(B)$; equivalently, $\cD(B)$ is the union of all decomposition elements $K \in \cD$ that intersect $B$. When the decomposition $\cD$ is implicit, we will speak simply of \emph{saturated} sets and \emph{saturations}.

Given a decomposition $\cD$ of a space $X$, the corresponding \emph{decomposition space} is the identification space $\rquotient{X}{\cD}$, equipped with the quotient topology. Alternatively, one can think of the decomposition space as the set $\cD$ itself, with the topology defined by declaring that $U \subset \cD$ is open whenever $\bigcup_{K \in U} K$ is open in $X$.

\begin{definition}[\hspace{1sp}{\cite[\S I.19]{Kuratowski1}}]\label{definition:USC}
	A decomposition $\cD$ of a space $X$ is \emph{upper semicontinuous} if it satisfies the following equivalent conditions:
	\begin{enumerate}
		\item The quotient map $X \to \rquotient{X}{\cD}$ is closed (i.e. the image of every closed set is closed).
		
		\item $U \subset X$ is open $\Rightarrow$ the union of all decomposition elements that are contained in $U$ is open.
		
		\item $A \subset X$ is closed $\Rightarrow$ the union of all decomposition elements that intersect $A$ is closed. Equivalently, the saturation of every closed set is closed.
	\end{enumerate}
\end{definition}

In a compact metric space, the upper semicontinuity of a decomposition can be understood in terms of the convergence properties of its elements.

\begin{theorem}[\hspace{1sp}{\cite[Theorem~IV.43.2]{Kuratowski2}}]\label{theorem:USCCM}
	Let $\cD$ be a decomposition of a compact metrizable space $X$. The following are equivalent:
	\begin{enumerate}
		\item $\cD$ is upper semicontinuous.
		
		\item If $K_1, K_2, \cdots \in \cD$ is a sequence of decomposition elements, and $\varliminf K_i$ intersects a decomposition element $K \in \cD$, then $\varlimsup K_i \subset K$.
		
		\item If $K_1, K_2, \cdots \in \cD$ is a convergent sequence of decomposition elements, then $\lim K_i$ is contained in some decomposition element $K \in \cD$.
	\end{enumerate}
\end{theorem}

\begin{lemma}[\hspace{1sp}{\cite[Theorem~IV.43.1]{Kuratowski2}}]\label{lemma:USCCH}
	Let $\cD$ be an upper semicontinuous decomposition of a compact Hausdorff space $X$. Then the decomposition space $\rquotient{X}{\cD}$ is compact Hausdorff.
\end{lemma}

The following easy lemmas are useful for constructing upper semicontinuous decompositions.

\begin{lemma}[\hspace{1sp}{\cite[Theorem~3-31]{HockingYoung}}]\label{lemma:PreimageUSC}
	If $f: X \to Y$ is a continuous map between compact metrizable spaces, then the decomposition $\{f^{-1}(y) \mid y \in Y\}$ by point-preimages is upper semicontinuous.
\end{lemma}

The \emph{monotonization} of a decomposition $\cD$ is the decomposition whose elements are the connected components of elements of $\cD$.

\begin{lemma}[\hspace{1sp}{\cite[Theorem~3-39]{HockingYoung}}]\label{lemma:MonotoneUSC}
	In a compact metrizable space, the monotonization of an upper semicontinuous decomposition is upper semicontinuous.
\end{lemma}

If $\cD$ is a nontrivial monotone upper semicontinuous decomposition of a closed interval $I$, then it's easy to see that the associated decomposition space $\rquotient{I}{\cD}$ is homeomorphic to a closed interval. The following theorem generalizes this to dimension $2$.

\begin{moores_theorem}[\hspace{1sp}\cite{Moore}]
	Let $\cD$ be a monotone upper semicontinuous decomposition of a closed $2$-disc $\mathbf{D}$, and suppose that each decomposition element $K \in \cD$ is nonseparating in $\mathbf{D}$. Then the decomposition space $\rquotient{\mathbf{D}}{\cD}$ is homeomorphic to a closed $2$-disc.
\end{moores_theorem}

This fails in higher dimensions.

\section{Leaves}\label{section:Leaves}
Fix, once and for all, a quasigeodesic flow $\Phi$ on a closed hyperbolic $3$-manifold $M$. We assume throughout that $M$ is orientable; this does not result in a loss of generality since passing to a double cover does not affect the existence of closed orbits.

We will work mostly with the lifted flow $\tPhi$ on the universal cover $\tM \simeq \bH^3$, whose orbits we call \emph{flowlines}. Since $M$ is compact, each nontrivial element $g \in \pi_1(M)$ acts as a loxodromic isometry on $\tM$, and has two fixed points in $S^2_\infty$ in an attracting-repelling pair. Since $M$ is closed, the action on $S^2_\infty$ is minimal, in the sense that every orbit is dense. See \cite{Thurston}.

The Morse Lemma implies that each flowline has well-defined and distinct endpoints in $S^2_\infty$. These endpoints vary continuously \cite[Lemma~4.3]{Calegari}, so we have a pair of $\pi_1(M)$-equivariant maps
\begin{align*}
	E^\pm: \tM &\to S^2_\infty \\
			x  &\mapsto \lim_{t \to \pm \infty} x \cdot t.
\end{align*}
Since quasigeodesics have distinct endpoints, we have
\[ E^+(x) \neq E^-(x) \text{ for all } x \in \tM. \]

Fix a point $z \in S^2_\infty$. Then $(E^+)^{-1}(z)$ is the union of all flowlines with positive endpoint at $z$, and each connected component of this set is called a \emph{weak positive leaf rooted at $z$}. Similarly, $(E^-)^{-1}(z)$ is the union of all flowlines with negative endpoint at $z$, each component of which is called a \emph{weak negative leaf rooted at $z$}.

The collections of all weak positive and weak negative leaves form a pair of monotone decompositions of $\tM$,
\begin{align*}
	&\cD^+_w := \{ \text{components of } (E^+)^{-1}(z) \mid z \in S^2_\infty \} \\
	\text{and}\hspace{.1in} &\cD^-_w := \{ \text{components of } (E^-)^{-1}(z) \mid z \in S^2_\infty \}.
\end{align*}
The action of $\pi_1(M)$ preserves these decompositions, so they descend under the covering map $\pi: \tM \to M$ to a pair of monotone partitions of $M$. These are only partitions, not decompositions, since the image of a weak leaf need not be closed.

A quasigeodesic flow on a closed hyperbolic manifold is always uniformly quasigeodesic \cite[Lemma~3.10]{Calegari}, so there is a uniform constant $C := C(k, \epsilon)$ such that the flowline through each $x \in \tM$ lies in the $C$-neighborhood of the geodesic from $E^-(x)$ to $E^+(x)$. We will use this in \S\ref{section:CoarseHyperbolicity} to see that the the flowlines in each weak positive/negative leaf are coarsely asymptotic in the forwards/backwards direction.

\subsection{The flowspace}\label{subsection:Flowspace}
Let $P$ be the orbit space of the lifted flow $\tPhi$, i.e. the set of flowlines in $\tM$ together with quotient topology induced by the map
\[ \nu: \tM \to P. \]
The deck group preserves the decomposition of $\tM$ into flowlines, so there is an induced action $\pi_1(M) \acts P$. The space $P$, together with this action, is called the \emph{flowspace} of $\Phi$.

Using uniform quasigeodesity, Calegari showed that $P$ is Hausdorff, and therefore homeomorphic to the plane \cite[Theorem~3.12]{Calegari}. The deck transformations preserve the orientation on the flowlines, as well as the orientation on the universal cover, so the flowspace action is orientation-preserving.

Although the flowspace is constructed as a quotient of the universal cover, we can use the following theorem to think of it as a transversal to the lifted flow.

\begin{theorem}[Montgomery-Zippin \cite{MontgomeryZippin}]\label{theorem:MontgomeryZippin}
	Let $\Psi$ be a flow on $\bR^3$ whose orbit space $P$ is Hausdorff. Then the quotient map $\nu: \bR^3 \to P$ admits a continuous section $\sigma: P \to \bR^3$.
\end{theorem}

Given such a section
\[ \sigma: P \to \tM, \]
the map
\begin{align*}
	\Sigma: P \times \bR &\to M 	\\
			(p, h) 		 &\mapsto \sigma(p) \cdot h
\end{align*}
is a homeomorphism that conjugates the ``vertical flow'' on $P \times \bR$, defined by $(p, h) \cdot t = (p, h + t)$, to the lifted flow $\tPhi$ on $\tM$.

We will write 
\[ \flow{p} := \nu^{-1}(p) \]
for the flowline corresponding to a point $p \in P$, and 
\[ \flow{A} := \nu^{-1}(A) \]
for the union of flowlines corresponding to a subset $A \subset P$. From the homeomorphism $\Sigma$ one sees that $\flow{A}$ is homeomorphic to $A \times \bR$ for any $A \subset P$.

\subsection{Leaves}
Let
\[ e^\pm: P \to S^2_\infty \]
be the maps that take each point $p \in P$ to the positive and negative endpoints of the corresponding flowline $\flow{p}$. These are just the factorizations of $E^\pm$ through the quotient map $\nu$, so they are continuous, $\pi_1(M)$-equivariant, and satisfy
\[ e^+(p) \neq e^-(p) \text{ for all } p \in P. \]
For each $z \in S^2_\infty$ the connected components of $(e^+)^{-1}(z)$ are called \emph{positive leaves rooted at $z$}, while the connected components of $(e^-)^{-1}(z)$ are called \emph{negative leaves rooted at $z$}. Equivalently, a positive or negative leaf rooted at $z$ is the image under $\nu$ of a weak positive or weak negative leaf rooted at $z$. See Figure~\ref{figure:Flowspace}.

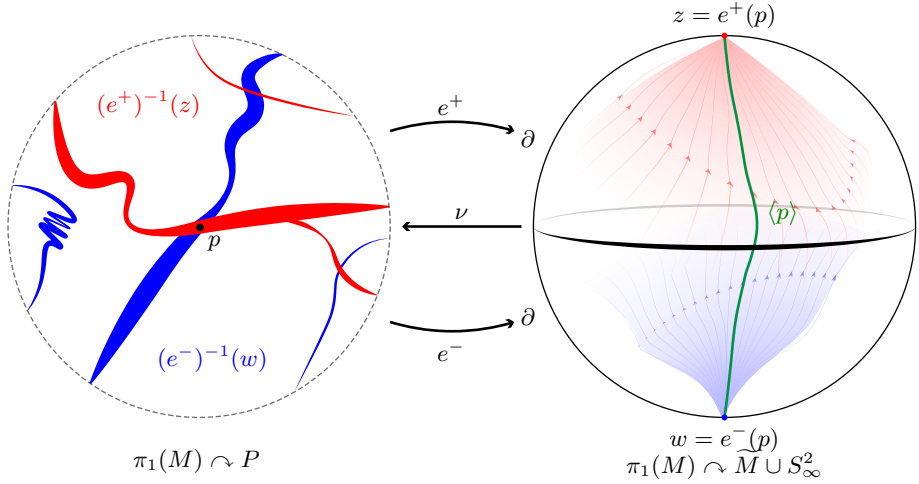
\begin{figure}[h]
	\vspace{.2in}
	\labellist
	\small\hair 2pt
	
	\pinlabel ${\pi_1(M) \acts P}$ [t] at 80 8
	\pinlabel $\textcolor{red}{(e^+)^{-1}(z)}$ [t] at 62 143
	\pinlabel $\textcolor{blue}{(e^-)^{-1}(w)}$ [t] at 86 47
	\pinlabel $p$ [tl] at 83 88
	
	\pinlabel $e^+$ [b] at 176 130
	\pinlabel $\partial$ at 205 123
	\pinlabel $\nu$ [b] at 180 91
	\pinlabel $e^-$ [t] at 176 49
	\pinlabel $\partial$ at 205 56
	
	\pinlabel ${\pi_1(M) \acts \tM \cup S^2_\infty}$ [t] at 279 8
	\pinlabel ${z = e^+(p)}$ [b] at 279 163
	\pinlabel ${w = e^-(p)}$ [t] at 279 17
	\pinlabel $\textcolor{dkgreen}{\flow{p}}$ [l] at 293 95
	
	\endlabellist
	\centering
	\includegraphics{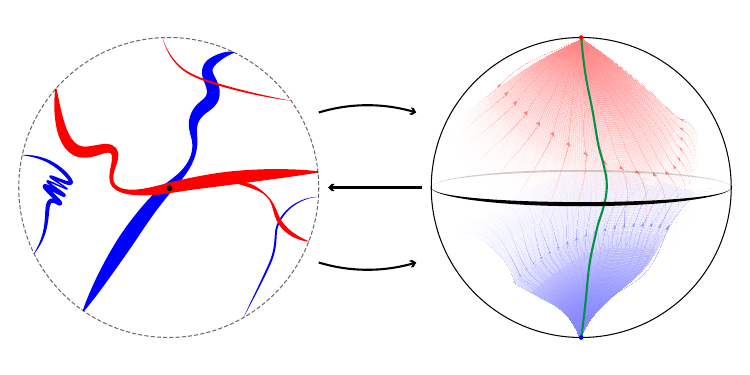}
	\caption{The positive and negative leaves through a point $p \in P$.} \label{figure:Flowspace}
\end{figure}

The collections of all positive and negative leaves form a pair of $\pi_1(M)$-invariant monotone decompositions of $P$, the \emph{positive} and \emph{negative decompositions}
\begin{align*}
							&\cD^+ := \{ \text{components of } (e^+)^{-1}(z) \mid z \in S^2_\infty \} \\
	\text{and}\hspace{.1in} &\cD^- := \{ \text{components of } (e^-)^{-1}(z) \mid z \in S^2_\infty \}.
\end{align*}
These have two important properties.

\begin{proposition}[\hspace{1sp}{\cite[Lemmas~4.8 \& 5.8]{Calegari}}]\
	\begin{enumerate}
		\item Each leaf $K \in \cD^+ \cup \cD^-$ is unbounded.
		\item If $K \in \cD^+$ is a positive leaf and $L \in \cD^-$ is a negative leaf, then $K \cap L$ is compact (and possibly empty).
	\end{enumerate}
\end{proposition}

We say that $\cD^+$ and $\cD^-$ are \emph{unbounded decompositions} that \emph{intersect compactly}. We will use these properties to treat the positive and negative decompositions as a broad generalization of a pair of foliations. Property (2) can be generalized as follows.

\begin{lemma}\label{lemma:BoundedSubsetsOfP}
	Let $A$ and $B$ be disjoint compact subsets of $S^2_\infty$. Then $(e^-)^{-1}(A) \cap (e^+)^{-1}(B)$ is compact.
\end{lemma}
\begin{proof}
	By uniform quasigeodesity, $(E^-)^{-1}(A) \cap (E^+)^{-1}(B)$ is contained in the $C$-neighborhood of the union of all geodesics from $A$ to $B$. Since $A$ and $B$ are compact and disjoint, we can find a compact set $D \subset \tM$ that intersects every one of these flowlines. Then $(e^-)^{-1}(A) \cap (e^+)^{-1}(B)$ compact since it is a closed subset of the compact set $\nu(D)$.
\end{proof}

\begin{remark}\label{remark:pA_FlowspaceFoliations}
	Let $\Psi$ be a pseudo-Anosov flow on a closed $3$-manifold $N$, not necessarily hyperbolic. Then the flowspace $\pi_1(N) \acts P$, defined in the same manner, is also a topological plane. The $2$-dimensional singular weak stable and unstable foliations lift to the universal cover and project to $1$-dimensional singular foliations of the flowspace $P$ which we'll call the \emph{stable and unstable foliations}. See \cite{Fenley}.
	
	If $N$ is hyperbolic, and $\Psi$ is quasigeodesic in addition to pseudo-Anosov, then the positive and negative decompositions that come from its quasigeodesic structure are exactly the stable and unstable foliations that come from its pseudo-Anosov structure.
\end{remark}

\subsection{Dynamics in the flowspace}\label{subsection:DynamicsInP}
Each point $p \in P$ corresponds to a flowline $\flow{p}$ in $\tM$, and an orbit $\pi(\flow{p})$ in $M$.

\begin{lemma}
	A point $p \in P$ corresponds to a closed orbit in $M$ if and only if there is a nontrivial element $g \in \pi_1(M)$ such that $g(p) = p$. Any such $g$ represents a multiple of the free homotopy class of the corresponding orbit.
\end{lemma}
\begin{proof}
	Let $g$ be a nontrivial element of $\pi_1(M)$ that fixes $p$. Then $g$ fixes the corresponding flowline $\flow{p} \subset \tM$, so the image $\pi(\flow{p})$ is a closed orbit, and $g$ represents a multiple of its free homotopy class.
	
	If $\pi(\flow{p})$ is closed, then it is homotopically nontrivial, since it has a lift $\flow{p}$ that is homeomorphic to a line. Take a point $x \in \flow{p}$ as the basepoint for $\tM$, and $x_0 = \pi(x)$ as the basepoint for $M$. Then an element $g \in \pi_1(M, x_0)$ that represents its homotopy class is nontrivial, and fixes $\flow{p}$ and hence $p$.
\end{proof}

Thus the Closed Orbits Theorem reduces to showing that the flowspace action is not free.

A point $x \in M$ is called \emph{$\omega$-recurrent} (\emph{$\alpha$-recurrent}) if there is a sequence of times $t_i \to \infty$ (resp. $t_i \to -\infty$) such that $\lim_{i \to \infty} x \cdot t_i = x$. If $x, y \in M$ are in the same orbit, then $x$ is $\alpha$- or $\omega$-recurrent if and only if $y$ is, so we can speak of \emph{orbits} being $\alpha$- or $\omega$-recurrent. A \emph{recurrent} point or orbit is one that is either $\alpha$- or $\omega$-recurrent. A point $p \in P$ is said to be \emph{recurrent}, \emph{$\alpha$-recurrent}, or \emph{$\omega$-recurrent} when this holds for the corresponding orbit in $M$.

\begin{lemma}
	A point $p \in P$ is recurrent if and only if there is a sequence of nontrivial elements $g_1, g_2, \cdots \in \pi_1(M)$ such that $\lim_{i \to \infty} g_i(p) = p$.
\end{lemma}
\begin{proof}
	Take a point $x \in \flow{p}$ as the basepoint for $\tM$, and $x_0 = \pi(x)$ as the basepoint for $M$.
	
	If $\pi(\flow{p})$ is $\omega$-recurrent, then we can find a sequence of times $t_i \to \infty$ such that $\lim_{i \to \infty} x_0 \cdot t_i = x_0$. For each $i$, let $g_i$ be the element of $\pi_1(M, x_0)$ that represents the almost-cycle obtained by concatenating the flow segment $x_0 \cdot [0, t_i]$ with an arc $c_i$ of shortest possible length. The corresponding lift is the concatenation of $x \cdot[0, t_i]$ with the lift $\widetilde{c}_i$ of $c_i$ that starts at $x \cdot t_i$, and $g_i$ takes the terminal endpoint of $\widetilde{c}_i$ to $x$. The length of $\widetilde{c}_i$ goes to $0$, so we have $\lim_{i \to \infty} g_i(x \cdot t_i) = x$, and hence $\lim_{i \to \infty} g_i(p) = p$. A similar argument applies if $p$ is $\alpha$-recurrent.
		
	On the other hand, suppose that we have a sequence $(g_i)_{i = 1}^\infty$ of nontrivial elements of $\pi_1(M)$ such that $\lim_{i \to \infty} g_i(p) = p$. Then we can find a sequence of times $t_i$ such that $\lim_{i \to \infty} g_i(x \cdot t_i) = x$, which means that $\lim_{i \to \infty} x_0 \cdot t_i = x_0$. Since the injectivity radius of $M$ is bounded below, it follows that the $t_i$ are unbounded, so we can take a subsequence with either $t_i \to \infty$ or $t_i \to -\infty$. Thus $p$ is either $\omega$- or $\alpha$-recurrent.
\end{proof}

Such a sequence $(g_i)_{i = 1}^\infty$ is called a \emph{recurrence sequence} for the recurrent point $p$. Note that $p$ is $\omega$-recurrent ($\alpha$-recurrent) if and only if we can find a recurrence sequence with the property that $\lim_{i \to \infty} g_i(x \cdot t_i) = x$ for a point $x \in \flow{p}$ and a sequence of times $t_i \to \infty$ (resp. $t_i \to -\infty$). Such a sequence is called an \emph{$\omega$-sequence} (resp. \emph{$\alpha$-sequence}) for $p$. 

More generally, we will denote the $\omega$- and $\alpha$-limit sets of a point $x \in M$ by $\omega(x)$ and $\alpha(x)$. That is, $y \in \omega(x)$ ($y \in \alpha(x)$) if and only if there is a sequence of times $t_i \to \infty$ (resp. $t_i \to -\infty$) such that $\lim_{i \to \infty} x \cdot t_i = y$. As with recurrence, we can treat $\omega$- and $\alpha$-limit sets on the level of orbits. Note that a point or orbit is $\omega$- or $\alpha$-recurrent if and only if it is contained in its own $\omega$- or $\alpha$-limit set. In the flowspace, we write $q \in \omega(p)$ or $q \in \alpha(p)$ whenever this holds for the corresponding orbits in $M$.

\begin{lemma}
	Let $p, q \in P$. Then $q \in \omega(p) \cup \alpha(p)$ if and only if there is a sequence of nontrivial elements $g_1, g_2, \cdots \in \pi_1(M)$ such that $\lim_{i \to \infty} g_i(q) = p$.
\end{lemma}

The proof is similar to that of the preceding lemma, and we can extend the notions of $\omega$- and $\alpha$-sequences in the obvious way. In particular, $q \in \omega(p)$ if and only if there is a sequence $(g_i)_{i = 1}^\infty$ such that $\lim_{i \to \infty} g_i(x \cdot t_i) = y$ for points $x \in \flow{p}$ and $y \in \flow{q}$ and times $t_i \to \infty$. We call this an \emph{$\omega$-sequence} for $q \in \omega(p)$.

In the sequel, we will restrict attention to $\omega$-sequences, though all of our results have corresponding versions for $\alpha$-sequences.

\section{Sprigs}\label{section:Sprigs}
A \emph{compactification} of a space $X$ consists a compact space $Y$, together with an identification of $X$ with a dense subset of $Y$. In \cite{Frankel_qgflows}, we showed that any finite collection $\{\cD^i\}_{i = 1}^k$ of unbounded decompositions of a plane $P$ that intersect compactly determines a \emph{universal compactification} to a closed disc $\hP = P \sqcup S^1_u$ with interior $P$ and boundary circle $S^1_u$. This may be characterized by the following properties:
\begin{enumerate}
	\item the closure of each $K \in \bigcup_i \cD^i$ in $\hP$ intersects the boundary circle in a nontrivial totally disconnected set $\overline{K} \cap S^1_u$,
	
	\item \[ \bigcup_{K \in \bigcup_i \cD^i} \overline{K} \cap S^1_u \]
	is dense in $S^1_u$,
	
	\item any other compactification with these properties is a quotient of $\hP$.
\end{enumerate}
It follows that any group action $\Gamma \acts P$ that preserves each $\cD^i$ extends uniquely to an action on the corresponding universal compactification.

In particular, this construction can be applied to the decompositions $\cD^\pm$ that come from a quasigeodesic flow (in which case the boundary circle, together with the restricted action, is identified with the universal circle constructed by Calegari in \cite{Calegari}). In \cite{Frankel_spherefilling}, we showed that the endpoint maps $e^\pm: P \to S^2_\infty$ extend uniquely to $\pi_1(M)$-equivariant maps $\he^\pm: \hP \to S^2_\infty$. These agree on the boundary circle, in the sense that $\he^+(s) = \he^-(s)$ for all $s \in S^1_u$.

\subsection{The compactified flowspace}
For convenience, we will work with the variant of this compactification provided by the following theorem.

\begin{theorem}\label{theorem:CompactifiedFlowspace}
	There is a compactification of $P$ to a closed disc $\hP = P \sqcup S^1_u$ with the following properties:
	\begin{enumerate}		
		\item the flowspace action $\pi_1(M) \acts P$ extends uniquely to an action $\pi_1(M) \acts \hP$;
		
		\item the endpoint maps $e^\pm: P \to S^2_\infty$ extend uniquely to $\pi_1(M)$-equivariant maps
		\[ \he^\pm: \hP \to S^2_\infty; \]
		
		\item the extended endpoint maps agree on the universal circle, i.e.
		\[ \he^+(s) = \he^-(s) \text{ for all } s \in S^1_u; \]
		and
		
		\item for each $z \in S^2_\infty$,
		\[ (\he^+)^{-1}(z) \cap S^1_u = (\he^-)^{-1}(z) \cap S^1_u\]
		is totally disconnected.
	\end{enumerate}
\end{theorem}
\begin{proof}
	In \cite[Theorem~7.9 \& Construction~5.7]{Frankel_qgflows} and \cite[Theorem~2.10 \& Proposition~2.11]{Frankel_spherefilling}, we constructed a compactification $\hP = P \sqcup S^1_u$ of $P$ with extended endpoint maps $\he^\pm: \hP \to S^2_\infty$ that satisfies properties (1)--(3). To obtain (4) we will pass to a quotient of this compactification.
	
	Let
	\[ \mathcal{C} := \{ \text{components of } (\he^+)^{-1}(z) \cap S^1_u \mid z \in S^2_\infty \} \cup  \{ \{p\} \mid p \in P \}.\]
	This is a monotone upper semicontinuous of $\hP$ by nonseparating subsets, so Moore's Theorem says that the decomposition space $\rquotient{\hP}{\mathcal{C}}$ is a closed disc.
	
	The decomposition $\mathcal{C}$ is trivial on the interior of $\hP$, so the interior of the quotient is still identified with $P$. It is preserved by the action of $\pi_1(M)$, so there is an induced action on the quotient. Each element of $\mathcal{C}$ maps to a single point under $\he^+$ and $\he^-$, so these descend to $\pi_1(M)$-equivariant maps on the quotient. Thus the theorem is satisfied after replacing $\hP$ by $\rquotient{\hP}{\mathcal{C}}$.
\end{proof}

The space $\hP$, together with the action of $\pi_1(M)$, is called the \emph{compactified flowspace}. The boundary circle $S^1_u$, together with the restricted action, is called the \emph{universal circle}.

Since the extended endpoint maps agree on the universal circle, we will denote their mutual restriction by
\[ \he: S^1_u \to S^2_\infty. \]
This is $\pi_1(M)$-equivariant, and since the action on $S^2_\infty$ is minimal, its image must be the entire sphere at infinity. This generalizes the Cannon-Thurston Theorem, which produces such equivariant sphere-filling curves for fibered hyperbolic $3$-manifolds \cite{CannonThurston}.

\subsection{Ends of leaves}
Given a subset $A \subset \hP$, we define $\partial A := A \cap S^1_u$. Given a positive or negative leaf $K \in \cD^\pm$, the points in $\partial \overline{K}$ are called \emph{ends} of $K$.\footnote{Our usage of the word ``end'' differs slightly from that of \cite{Frankel_qgflows} and \cite{Frankel_spherefilling}, where it refers to a \emph{Freudenthal end}. Each Freudenthal end of a leaf $K$ maps to a point in the universal circle, and $\partial \overline{K}$ is the closure of the image of $K$'s Freudenthal ends. We will not need this, but the reader may refer to \cite[Lemma~7.8]{Frankel_qgflows} for more details.}

\begin{lemma}
	Each leaf has a nontrivial and totally disconnected set of ends.
\end{lemma}
\begin{proof}
	Nontriviality follows from that fact that the leaves are unbounded subsets of $P$. The ends of a leaf rooted at $z$ are contained in the totally disconnected set $\he^{-1}(z)$, and hence totally disconnected.
\end{proof}

Although distinct positive leaves $K, L \in \cD^+$ are disjoint in the flowspace, their closures may intersect in the universal circle, so that they ``share an end'' $s \in \partial \overline{K} \cap \partial \overline{L}$. Negative leaves may also share ends, and a positive leaf may share an end with a negative leaf.

\begin{lemma}\label{lemma:LeavesSharingEndsRoots}
	Any two leaves (both positive, both negative, or one of each) that share an end are rooted at the same point.
\end{lemma}
\begin{proof}
	If $s$ is an end of a leaf $K \in \cD^\pm$, then $K$ is rooted at $\he(s)$. Indeed, if $K$ is a positive leaf then its root is $e^+(K) = \he^+(\overline{K}) = \he^+(s) = \he(s)$, and if $K$ is a negative leaf then its root is $e^-(K) = \he^-(\overline{K}) = \he^-(s) = \he(s)$. 
	
	Therefore, two leaves $K$ and $L$ that share an end $s \in \partial \overline{K} \cap \partial \overline{L}$ are both rooted at $\he(s)$.
\end{proof}

On the other hand, positive and negative cannot simultaneously intersect and share ends.

\begin{lemma}\label{lemma:LeavesSharingEndsDisjoint}
	If $K \in \cD^+$ and $L \in \cD^-$ share an end, then $K \cap L = \emptyset$.
\end{lemma}
\begin{proof}
	If $K$ and $L$ share an end $s \in \partial \overline{K} \cap \partial \overline{L}$, then we have $e^+(K) = \he(s) = e^-(L)$ by the preceding lemma. Then $K$ and $L$ must be disjoint, since any point $p \in K \cap L$ would have $e^+(p) = e^-(p)$, which contradicts the fact that quasigeodesics have distinct endpoints.
\end{proof}

\subsection{Sprigs}\label{subsection:Sprigs}
Since distinct positive/negative leaves can share ends, the collection of positive/negative leaf-closures does not form a decomposition of the compactified flowspace. In addition, there may be points in $S^1_u$ that are not in the closure of a positive or negative leaf. However, there are natural decompositions of $\hP$ obtained using the extended endpoint maps.

For each $z \in S^2_\infty$, the connected components of $(\he^+)^{-1}(z)$ and $(\he^-)^{-1}(z)$ are called, respectively, \emph{positive sprigs rooted at $z$} and \emph{negative sprigs rooted at $z$}. The collections of all positive and negative sprigs form a pair of $\pi_1(M)$-invariant monotone decompositions of $\hP$, the positive and negative \emph{sprig decompositions}
\begin{align*}
	\hcD^+ &:= \{ \text{components of } (\he^+)^{-1}(z) \mid z \in S^2_\infty \} \\
	\hcD^- &:= \{ \text{components of } (\he^-)^{-1}(z) \mid z \in S^2_\infty \}.
\end{align*}

Let $\hK \in \hcD^\pm$ be a positive or negative sprig. Then the points in $\partial \hK = \hK \cap S^1_u$ are called \emph{ends} of $\hK$, while $\chK := \hK \cap P$ is called its \emph{bounded part}. Note that
\[ \hK = \chK \sqcup \partial \hK. \]

\begin{lemma}\label{lemma:EndsOfSprigs}
	Each sprig has a nontrivial and totally disconnected set of ends.
\end{lemma}

The proof is the same as in the case of leaves. A sprig $\hK \in \hcD^\pm$ is said to be \emph{trivial} if $\chK = \emptyset$, and \emph{nontrivial} otherwise.

\begin{lemma}\label{lemma:TrivialNontrivialSprig}
	Let $\hK \in \hcD^\pm$ be a positive/negative sprig rooted at $z$.
	\begin{enumerate}
		\item If $\hK$ is trivial, then it consists a single point in $S^1_u$.
		
		\item If $\hK$ is nontrivial, then each component of $\chK$ is a positive/negative leaf rooted at $z$.
	\end{enumerate}
\end{lemma}
\begin{proof}
	Suppose without loss of generality that $\hK$ is a positive sprig. If $\hK$ is trivial, then $\hK = \partial \hK$ it is a connected subset of the totally disconnected set $\he^{-1}(z)$, and hence a single point. If $\hK$ is nontrivial, then each component of $\chK$ is a connected component of $(\he^+)^{-1}(z) \cap S^1_u = (e^+)^{-1}(z)$, which is a a positive sprig rooted at $z$.
\end{proof}

On the other hand:

\begin{lemma}
	Any two positive leaves that share an end are contained in the same positive sprig, and any two negative leaves that share an end are contained in the same negative sprig.
\end{lemma}
\begin{proof}
	If $K$ and $L$ are positive leaves that share an end $s \in \partial \overline{K} \cap \partial \overline{L}$, then $\overline{K} \cup \overline{L}$ is connected. By Lemma~\ref{lemma:LeavesSharingEndsRoots}, $e^+(K \cup L)$ is a single point, so $\he^+(\overline{K} \cup \overline{L})$ is a single point. The same argument holds for negative sprigs using $e^-$ and $\he^-$.
\end{proof}

Distinct positive sprigs are disjoint by definition, as are distinct negative sprigs. However, a positive sprig might share an end with a negative sprig, and we have the following analogue of Lemmas~\ref{lemma:LeavesSharingEndsRoots} and \ref{lemma:LeavesSharingEndsDisjoint}.

\begin{lemma}\label{lemma:SprigsSharingEndsDisjoint}
	Let $\hK \in \hcD^+$ and $\hL \in \hcD^-$ be positive and negative sprigs with $\partial \hK \cap \partial \hL \neq \emptyset$. Then $\he^+(\hK) = \he^-(\hL)$, and $\chK \cap \chL = \emptyset$
\end{lemma}

Conversely, sprigs whose bounded parts intersect must have disjoint ends, which we express as follows.

\begin{corollary}\label{corollary:DisjointEnds}
	For each $p \in P$, we have $\partial \hcD^+(p) \cap \partial \hcD^-(p) = \emptyset$.
\end{corollary}

Here, we are using the notation from \S\ref{subsection:Decompositions}: $\hcD^+(p)$ and $\hcD^-(p)$ are the $\hcD^+$- and $\hcD^-$-saturations of the set $\{p\}$, which are just the positive and negative sprigs through $p$.

As illustrated in Figure~\ref{figure:SequenceOfLeaves}, one can have a sequence of positive leaves that limits on more than one positive leaf. Note, however, that the two limit leaves share an end, and are therefore contained in a single positive sprig. This is a result of the following lemma, which is an immediate consequence of Lemmas~\ref{lemma:PreimageUSC} and \ref{lemma:MonotoneUSC}.

\begin{figure}[h]
	\vspace{.2in}
	\labellist
	\small\hair 2pt

	\endlabellist
	\centering
	\includegraphics{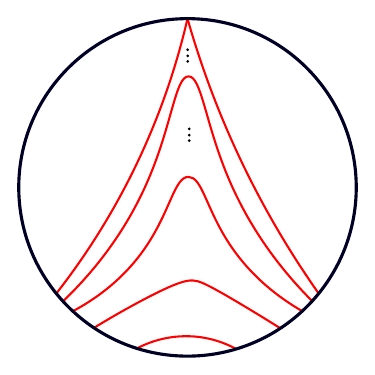}
	\caption{Leaves that split in the limit.} \label{figure:SequenceOfLeaves}
\end{figure}

\begin{lemma}\label{lemma:SprigUSC}
	The sprig decompositions $\hcD^+$ and $\hcD^-$ are upper semicontinuous.
\end{lemma}

In particular, the sprig decompositions have the following properties from Definition~\ref{definition:USC} and Theorem~\ref{theorem:USCCM}.
\begin{enumerate}
	\item Let $\hK_1, \hK_2, \cdots$ be a sequence of positive (resp. negative) sprigs for which $\varliminf \hK_i \neq \emptyset$. Then $\varlimsup \hK_i$ is contained in a single positive (negative) sprig.
	\item Let $p_1, p_2, \cdots$ be a convergent sequence of points in $\hP$. Then $\varlimsup \hcD^\pm(p_i) \subset \hcD^\pm(\lim p_i)$.
	\item If $A \subset \hP$ is closed, then the saturations $\hcD^+(A)$ and $\hcD^-(A)$ are closed.
\end{enumerate}

We can abstract the properties of the sprig decompositions as follows.

\begin{definition}\label{definition:SpideryDecomposition}
	A decomposition $\hcD$ of the closed disc $\mathbf{D} = D \sqcup S^1$ is \emph{spidery} if it is upper semicontinuous, and each decomposition element $\hK \in \hcD$ intersects the boundary circle in a nontrivial totally disconnected set $\hK \cap S^1$.
\end{definition}

We do not yet have a satisfactory notion of transversality between the spidery decompositions $\hcD^\pm$. This is done in \S\ref{subsection:SpideryPair}.

\section{Master sprigs}\label{section:MasterSprigs}
Using the compactified flowspace, we will construct a compactification of $\tM$, called the \emph{flow ideal compactification}, that is especially adapted to the lifted flow. The boundary of this compactification is a $2$-sphere $S^2_u$, called the \emph{universal sphere}, which we will use to understand the relationship between positive and negative sprigs.

\begin{remark}
	Cannon-Thurston constructed flow ideal compactifications for pseudo-Anosov suspension flows \cite{CannonThurston} and Fenley generalized this to all pseudo-Anosov flows \cite{Fenley}.
\end{remark}

\subsection{Master sets}
For each $z \in S^2_\infty$, the set
\[ Z = (\he^+)^{-1}(z) \cup (\he^-)^{-1}(z) \subset \hP, \]
which consists of all positive and negative sprigs rooted at $z$, is called the \emph{master set rooted at $z$}. 

\begin{lemma}\label{lemma:MasterSetsConnected}
	Each master set is connected.
\end{lemma}
\begin{proof}
	Fix a master set $Z$ rooted at a point $z \in S^2_\infty$, and let $D_1 \supset D_2 \supset \cdots$ be a nested sequence of open discs in $S^2_\infty$ centered at $z$, with $\bigcap_{i = 1}^\infty D_i = z$. For each $i$, let $H_i \subset \tM$ be the union of all geodesics with both endpoints in $D_i$.
	
	Each $H_i$ projects to a connected subset $X_i := \nu(H_i)$ of the flowspace, where $p \in X_i$ if and only if $\flow{p}$ intersects $H_i$. The closures $\overline{X}_i$ are compact connected subsets of the compactified flowspace, so $\bigcap_i \overline{X}_i$ is compact and connected. To complete the proof, we will show that $\bigcap_i \overline{X}_i = Z$.
	
	Let us show that $Z \subset \bigcap_i \overline{X}_i$. Let $p \in Z$. If $p \in Z \cap P$, then either $e^+(p) = z$ or $e^-(p) = z$, so we have $p \in X_i \subset \overline{X}_i$ for every $i$, and hence $p \in \bigcap_i \overline{X}_i$. If $p \in Z \cap S^1_u$, then $\he^+(p) = z = \he^-(p)$. Take a sequence of points $p_j \in P$ that converge to $p$, and note that $\lim e^+(p_j) = \he^+(p) = z$. Then after taking a subsequence of these points, we can assume that $e^+(p_j) \subset D_i$ for all $j \geq i$, which means that $p_j \in X_i$ for all $j \geq i$. Then $p = \lim_{j \to \infty} p_j \in \overline{X}_i$ for all $i$, and hence $p \in \bigcap_i \overline{X}_i$.
	
	Now we will show that $\hP \setminus Z$ is disjoint from $\bigcap_i \overline{X}_i$. Observe that if $A \subset S^2_\infty$ is compact, and $z \notin A$, then union of all flowlines with both ends in $A$ is eventually disjoint from $H_i$, since it is contained in the $C$-neighborhood of union of all geodesics with both endpoints in $A$. Alternatively, if $B \subset P$ has the property that $e^+(B) \cup e^-(B)$ is bounded away from $z$, then we eventually have $B \cap X_i = \emptyset$.
	
	Let $p$ be a point in $\hP \setminus Z$, which means that $\he^+(p) \neq z \neq \he^-(p)$. If $p \in P$, then this is the same as $e^+(p) \neq z \neq e^-(p)$, so we can find an open neighborhood $U \subset P$ of $p$ such that $e^+(U) \cup e^-(U)$ is bounded away from $z$, and hence $U$ is eventually disjoint from $X_i$. Then $U$ is eventually disjoint from $\overline{X}_i$, and hence $p \notin \bigcap_i \overline{X}_i$. If $p \in S^1_u$, then we can find an open neighborhood $U \subset \hP$ of $p$ such that $\he^+(U) \cup \he^-(U)$ is bounded away from $z$. Then $e^+(U \cap P) \cup e^+(U \cap P)$ is bounded away from $z$, so $U \cap P$ is eventually disjoint from $X_i$. Any sequence of points in $P$ that approaches a point in $U \cap S^1_u$ is eventually contained in $U \cap P$, so $U$ is eventually disjoint from $\overline{X}_i$, and hence $p \notin \overline{X}$.
\end{proof}

Let $Z$ be a master set rooted at a point $z$. As with sprigs, we write $Z = \mathring{Z} \sqcup \partial Z$, where the points in $\partial Z = Z \cap S^1_u$ are called \emph{ends} of $Z$ and $\mathring{Z} := Z \cap P$ is called its \emph{bounded part}. Since $\partial Z = \he^{-1}(z)$, each master sprig has a nontrivial, closed, and totally disconnected set of ends. We say that $Z$ is \emph{trivial} if $\mathring{Z} = \emptyset$, in which case $Z = \partial Z$ consists of a single point, being a connected subset of a totally disconnected set. 

Note, however, that the master sets do not form a decomposition of $\hP$, since each point $p \in P$ is contained in two master sets, those rooted at $e^+(p)$ and $e^-(p)$.

\subsection{The flow ideal compactification}
Recall from \S\ref{subsection:Flowspace} that we can identify $\tM$ with the open cylinder
\[ C := P \times \bR \]
in such a way that the flowlines of $\tPhi$ correspond to vertical lines $\{p\} \times \bR$ in $C$. We can compactify $C$ by thinking of it as the interior of the closed cylinder
\[ \mathbf{C} := \hP \times \overline{\bR}, \]
where $\overline{\bR} = [-\infty, \infty]$ is the usual two-point compactification of $\bR = (-\infty, \infty)$. The action of $\pi_1(M)$ on $C \simeq \tM$ can be extended to the upper and lower horizontal faces
\[ \partial_\pm \mathbf{C} := \hP \times \{\pm \infty\} \]
of the closed cylinder, since they are identified with the compactified flowspace, but not to the vertical face
\[ \partial_v \mathbf{C} := S^1_u \times \overline{\bR}. \]
However, it does extend to the ``closed lens''
\[ \mathbf{L} := \rquotient{\mathbf{C}}{\{ \{s\} \times \overline{\bR} \mid s \in S^1_u\}} \]
obtained by collapsing the vertical lines in the vertical face. This space, together with the action $\pi_1(M) \acts \mathbf{L}$, is called the \emph{flow ideal compactification}.

Let $S^2_u$ denote the boundary sphere of $\mathbf{L}$, which, together with the action of $\pi_1(M)$, is called the \emph{universal sphere}. This consists of two copies of the compactified flowspace glued along their universal circles. We will denote these by $\hP^+$ and $\hP^-$, and think of them as the upper and lower hemispheres of $S^2_u$. 

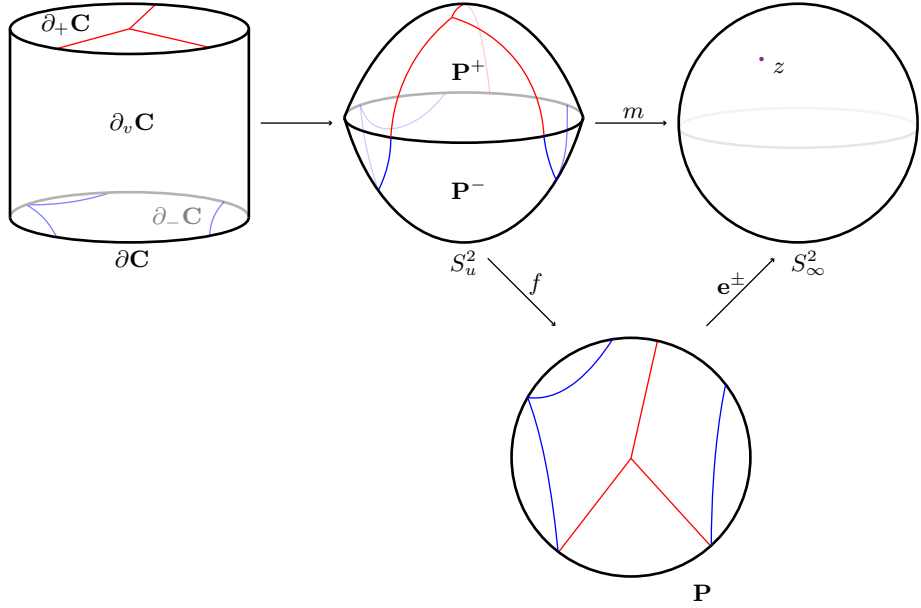
\begin{figure}[h]
	\vspace{.2in}
	\labellist
	\small\hair 2pt
	
	\pinlabel ${\partial_+ \mathbf{C}}$ [c] at 30 217
	\pinlabel ${\partial_v \mathbf{C}}$ [c] at 55 180
	\pinlabel $\textcolor{gray}{\partial_- \mathbf{C}}$ [c] at 72 144
	\pinlabel ${\partial \mathbf{C}}$ [t] at 55 134
	
	\pinlabel $\hP^+$ [c] at 182 200
	\pinlabel $\hP^-$ [c] at 182 155
	\pinlabel $S^2_u$ [t] at 180 134
	
	\pinlabel $m$ [b] at 244 180
	
	\pinlabel $S^2_\infty$ [t] at 310 134
	\pinlabel $z$ [tl] at 295 204
	
	\pinlabel $f$ [bl] at 203 114
	
	\pinlabel $\hP$ [t] at 270 8
	
	\pinlabel $\he^\pm$ [br] at 288 114
	
	\endlabellist
	\centering
	\includegraphics{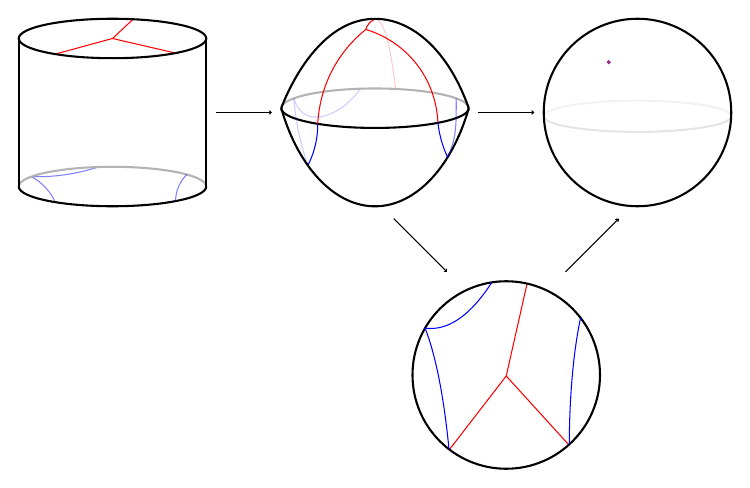}
	\caption{The universal sphere.} \label{figure:FlowIdealBoundary}
\end{figure}

\subsection{Master sprigs}
We can think of the extended endpoint maps as maps 
\[ \he^\pm: \hP^\pm \to S^2_\infty \]
supported on the two hemispheres of the universal sphere. These agree on the equator, so they combine to form a $\pi_1(M)$-equivariant map
\[ m: S^2_u \to S^2_\infty. \]
Alternatively, the identifications $\hP^\pm \simeq \hP$ define a \emph{flattening map}
\[ f: S^2_u \to \hP, \]
and we can define $m$ by
\[ m(p) =
\begin{cases}
\he^+ \circ f(p), & \text{for $p \in \hP^+$}\\
\he^- \circ f(p), & \text{for $p \in \hP^-$}.
\end{cases}
\]

For each $z \in S^2_\infty$, the set $m^{-1}(z) \subset S^2_u$ consists of all positive sprigs rooted at $z$ lying in the upper hemisphere, together with all negative sprigs rooted at $z$ lying in the lower hemisphere. We call this the \emph{master sprig rooted at $z$}. Note that the flattening map $f$ takes each master sprig to the master set rooted at the same point. See Figure~\ref{figure:FlowIdealBoundary}.

Although the master \emph{sets} do not form a decomposition of $\hP$, the master \emph{sprigs} form a decomposition of $S^2_u$, the \emph{master decomposition}
\[ \cM := \{ m^{-1}(z) \mid z \in S^2_\infty \}. \]
This is upper semicontinuous by Lemma~\ref{lemma:PreimageUSC}, and monotone by the following lemma.

\begin{lemma}
	Each master sprig is connected.
\end{lemma}
\begin{proof}
	Suppose that some master sprig $Z = m^{-1}(z)$ is disconnected. Then we can write it as a disjoint union
	\[ Z = A \sqcup B \]
	of nontrivial compact sets.	Since $f$ restricts to a homeomorphism from the equator to $S^1_u$, and takes only points in the equator to $S^1_u$, it follows that $f(A) \cap S^1_u$ is disjoint from $f(B) \cap S^1_u$. Furthermore, $f(A) \cap P$ is disjoint from $f(B) \cap P$, since $p \in f(A) \cap f(B) \cap P$ would imply that $e^+(p) = e^-(p)$, which contradicts the fact that quasigeodesics have distinct endpoints. Thus we can write the corresponding master set as a disjoint union of compact sets
	\[ f(Z) = f(A) \sqcup f(B), \]
	which contradicts Lemma~\ref{lemma:MasterSetsConnected}.
\end{proof}

\subsection{Recovering the sphere at infinity}
The map $m: S^2_u \to S^2_\infty$ takes each master sprig to a point, so it factors through a $\pi_1(M)$-equivariant map
\[ \hatm: \rquotient{S^2_u}{\cM} \to S^2_\infty \]
defined on the decomposition space of $\cM$.

\begin{lemma}
	$\hatm$ is a homeomorphism.	
\end{lemma}
\begin{proof}
	Since $\cM$ is upper semicontinuous, the decomposition space $\rquotient{S^2_u}{\cM}$ is compact Hausdorff by Lemma~\ref{lemma:USCCH}. Since $\cM$ is monotone, and $m$ is surjective, the map $\hatm$ is bijective, and a continuous bijection between compact Hausdorff spaces is a homeomorphism.
\end{proof}

Thus we recover the sphere at infinity as a quotient of the universal sphere. As a consequence, we have the following observation.

\begin{lemma}\label{lemma:MasterSprigsNonseparating}
	Each master sprig is nonseparating in $S^2_u$.
\end{lemma}
\begin{proof}
	Suppose that some master sprig $Z = m^{-1}(z)$ separates $S^2_u$. Choose some complementary component $U$ of $Z$, and let $V$ be the union of all other complementary components. Then $U$ and $V$ are open unions of master sprigs, so they map to disjoint open sets $\hatm(U), \hatm(V) \subset S^2_\infty$, and $S^2_\infty = \hatm(U) \sqcup z \sqcup \hatm(V)$. But this means that $z$ is a cutpoint of $S^2_\infty$, which is impossible.
\end{proof}

\subsection{Spidery pairs}\label{subsection:SpideryPair}
Let $\hK \in \hcD^+$ and $\hL \in \hcD^-$ be positive and negative sprigs. If $\hK$ and $\hL$ share more than one end, then we can think of them as forming a kind of ``ideal bigon'' in $\hP$. The following lemma says that this cannot happen. One also show that there are no ``ideal polygons.'' That is, one cannot have an alternating sequence of positive and negative sprigs $\hK_0, \hK_1, \hK_2, \cdots, \hK_{n-1}$ where each $\hK_i$ shares an end with $\hK_{i+1}$ (mod $n$). We leave this as an exercise since it will not be used directly.

\begin{lemma}[No bigons]\label{lemma:NoBigons}
	Let $\hK \in \hcD^+$ and $\hL \in \hcD^-$ be positive and negative sprigs. Then $\partial \hK$ intersects $\partial \hL$ in at most one point.
\end{lemma}
\begin{proof}
	We will use following fact from classical analysis situs \cite[Theorem~II.5.28a]{Wilder}: If $A$ and $B$ are compact connected subsets of $S^2$, and $A \cap B$ is disconnected, then $A \cup B$ separates $S^2$.
	
	Think of $\hK$ as a subset of $\hP^+$ and $\hL$ as a subset of $\hP^-$. If they intersect in $S^1_u$ then they are contained in a single master sprig. If they intersect at more than one point in $S^1_u$, then this master sprig is separating, contradicting the preceding lemma. 
\end{proof}

It follows that the sprig decompositions satisfy the following definition.

\begin{definition}\label{definition:SpideryPair}
	A \emph{spidery pair} consists of two spidery decompositions $\hcD^\pm$ of the closed disc $\mathbf{D} = D \sqcup S^1$ with the following property: for each $\hK \in \hcD^+$ and $\hL \in \hcD^-$, the intersection $\hK \cap \hL$ is either empty, a compact subset of $D$, or a single point in $S^1$.
\end{definition}

\subsection{Fixed sprigs}
As with sprigs and master sets, a master sprig is said to be \emph{trivial} if it is contained entirely in $S^1_u$, thought of as the equator of $S^2_u$. Equivalently, a master sprig $m^{-1}(z)$ is trivial if and only if the corresponding master set $f(m^{-1}(z))$ is trivial. We will use the following lemma to simplify the process of finding closed orbits.

\begin{lemma}\label{lemma:ElementPreservingLeaf}
	Let $g$ be a nontrivial element of $\pi_1(M)$ that fixes a leaf, nontrivial sprig, nontrivial master set, or nontrivial master sprig. Then $g$ fixes a point in $P$.
\end{lemma}
\begin{proof}
	We will use the Brouwer Plane Translation Theorem (cf. \cite{Franks}), which says that an orientation-preserving homeomorphism of the plane with a bounded forward orbit must have a fixed point.
	
	If $g$ fixes a leaf, nontrivial sprig, or nontrivial master sprig, then it fixes the corresponding master set, which is nontrivial. Thus we can assume that $g$ fixes a nontrivial master set $Z$.
	
	In the sphere at infinity $S^2_\infty$, $g$ has exactly two fixed points $\alpha$ and $\rho$ in an attracting-repelling pair, so $Z$ must be rooted at one of these. After possibly replacing $g$ by its inverse, we can assume that $Z$ is rooted at the repelling fixed point $\rho$.
	
	Since $Z$ is nontrivial, we can choose a point $p \in Z \cap P$. If $p$ is contained in a positive subleaf of $Z$, then $e^+(p) = \rho$ and $e^-(p) \neq \rho$, so we have $\lim_{i \to \infty} e^+(g^i(p)) = \rho$ and $\lim_{i \to \infty} e^-(g^i(p)) = \alpha$. Similarly, if $p$ is contained in a negative subleaf, then $\lim_{i \to \infty} e^-(g^i(p)) = \rho$ and $\lim_{i \to \infty} e^+(g^i(p)) = \alpha$.
	
	Either way, this means that $\lim_{i \to \infty} e^+(g^i(p)) \neq \lim_{i \to \infty} e^-(g^i(p))$. Since $\he^+$ agrees with $\he^-$ on $S^1_u$, this means that $g^i(p)$ stays in a bounded subset of $P$ for all $i \geq 0$. Then $g$ fixes some point in $P$ by the Brouwer Plane Translation theorem, as does $g^{-1}$.
\end{proof}

\section{Decompositions I: Separation}\label{section:Separation}
In this section we study the structure of the individual sprig decompositions, and relate the separation properties of sprigs with those of their ends. In particular, we will see that any two positive or negative sprigs are separated from each other by an interval's worth of positive or negative sprigs. The relationship between the positive and negative sprig decompositions is covered in the following section.

Throughout this section, take $\hcD$ to be either $\hcD^+$ or $\hcD^-$, and ``sprig'' to mean a positive or negative sprig accordingly. In fact, the results in this section apply to any spidery decomposition of the closed disc (Definition~\ref{definition:SpideryDecomposition}).

\begin{lemma}\label{lemma:TwoEndedSprigs}
	Some sprig $\hK \in \hcD$ has at least two ends.
\end{lemma}
\begin{proof}
	Suppose that each sprig has exactly one end. Then we can define a map
	\[ r: \hP \to S^1_u \]
	that sends each point $p \in \hP$ to the end $r(p)$ of its sprig. This is continuous by upper semicontinuity, and it restricts to the identity on $S^1_u$. That is, it is a retraction of $\hP$ onto $S^1_u$, which is impossible.
\end{proof}

\subsection{Complementary regions and intervals}
Fix an orientation on $S^1_u$. Then any ordered pair of points $s, t \in S^1_u$ determines an oriented open subinterval $(s, t) \subset S^1_u$ running from $s$ to $t$. This makes sense even when $s = t$, where we take $(s, s) = S^1_u \setminus \{s\}$.

Given a closed subset $A \subset \hP$, the connected components of $\hP \setminus A$ are called \emph{complementary regions of $A$}, while the connected components of $S^1_u \setminus \partial A$ are called \emph{complementary intervals of $\partial A$} (or of $A$). These are open intervals whose initial and terminal points lie in $\partial A$.

\begin{lemma}\label{lemma:SprigComplementaryRegions}
	Let $U$ be a complementary region of a sprig $\hK \in \hcD$. Then $U \cap S^1_u$ is a complementary interval of $\partial \hK$.
\end{lemma}
\begin{proof}
	Choose an arbitrary point $p \in U$. Then the sprig $\hcD(p)$ is contained in $U$, so $U \cap S^1_u$ is nontrivial since it contains $\partial \hcD(p)$. Each complementary interval of $\partial \hK$ that intersects $U \cap S^1_u$ is contained in it, so $U \cap S^1_u$ is a nontrivial union of complementary intervals.
	
	Suppose that $U \cap S^1_u$ contains two distinct complementary intervals $I$ and $J$. Since $U$ is path-connected (it is a connected open subspace of a locally path-connected space), we can find an arc $c \subset U$ with initial point in $I$ and terminal point in $J$. This separates the endpoints of $I$, which are contained in $\hK$, so it separates $\hK$. But $\hK$ is connected, so $U \cap S^1_u$ must be a single complementary interval.
\end{proof}

Since each complementary interval is contained in a complementary region, it follows that $U \mapsto U \cap S^1_u$ defines a bijection between the complementary regions of a sprig and the complementary intervals of its ends. In particular, an $n$-ended sprig $\hK$ has exactly $n$ complementary regions. 

\begin{corollary}\label{corollary:SprigsUnlinked}
	Let $\hK, \hL \in \hcD$ be distinct. Then $\partial \hK$ is contained in a single complementary interval of $\partial \hL$.
\end{corollary}

Note that this relies on the fact that sprigs are closed. For example, \cite{Dawson} contains an illustration of two disjoint, connected, non-closed subsets of the square that connect opposite pairs of corners.

The following easy corollary says that the separation properties of sprigs can be seen in terms of their ends.

\begin{corollary}\label{corollary:SeparatingSprigsViaEnds}
	Let $\hK, \hK', \hK'' \in \hcD$ be distinct. Then $\hK'$ separates $\hK$ from $\hK''$ if and only if $\partial \hK'$ separates $\partial \hK$ from $\partial \hK''$.
\end{corollary}

\subsection{Saturated continua}
The ($\hcD$)-saturation of a connected set is connected, and the saturation of a closed set is closed by upper semicontinuity, cf. \S\ref{subsection:Decompositions}. A closed, connected, saturated subset of is called a \emph{saturated continuum}.

The preceding results generalize easily as follows.

\begin{lemma}\label{lemma:SaturatedComplementaryRegions}
	Let $U$ be a complementary region of a saturated continuum $\hA \subset \hP$. Then $U \cap S^1_u$ is a complementary interval of $\partial \hA$.
\end{lemma}

\begin{corollary}\label{corollary:SaturatedsUnlinked}
	Let $\hA, \hB \subset \hP$ be disjoint saturated continua. Then $\partial \hA$ is contained in a single complementary interval of $\partial \hB$.
\end{corollary}

\begin{corollary}\label{corollary:SeparatingSaturatedsViaEnds}
	Let $\hA, \hA', \hA'' \subset \hP$ be disjoint saturated continua. Then $\hA'$ separates $\hA$ from $\hA''$ if and only if $\partial \hA'$ separates $\partial \hA$ from $\partial \hA''$.
\end{corollary}

Let $\hA \subset \hP$ be a saturated continuum. For each complementary region $U$ of $\hA$, we define the corresponding \emph{face} of $\hA$ by
\[ F(U) := \overline{U} \cap \hA. \]

\begin{lemma}\label{lemma:FaceInSprig}
	Let $U$ be a complementary region of a saturated continuum $\hA \subset \hP$. Then $F(U)$ is a connected subset of a single sprig.
\end{lemma}
\begin{proof}
	The face $F(U)$ is connected because the disc has the \emph{Brouwer property} \cite[\S~II.4]{Wilder}: If $X \subset \hP$ is closed and connected, and $U$ is a component of $\hP \setminus X$, then the frontier of $U$ is closed and connected.
	
	Let $a, b \in \partial \hA$ be the endpoints of the corresponding complementary interval $U \cap S^1_u = (a, b)$, and note that $a, b \in F(U)$. Let $\hK_a = \hcD(a)$ and $\hK_b = \hcD(b)$ be the sprigs through these points. We will show that $F(U) \subset \hK_a \cup \hK_b$. Since $F(U)$ is connected, this implies that $\hK_a = \hK_b$, which completes the lemma.
	
	Fix a point $p \in F(U)$, and choose a sequence of points $p_1, p_2, \cdots \in U$ that converge to $p$. The ends of each $\hcD(p_i)$ are contained in $(a, b)$, so they must accumulate on either $a$ or $b$ because $\varlimsup \hcD(p_i) \subset \hcD(p) \subset A$. Thus we have either $\hcD(p) = \hK_a$ or $\hcD(p) = \hK_b$. This applies for every point in $F(U)$, so $F(U) \subset \hK_a \cup \hK_b$ as desired.
\end{proof}

\subsection{Separating sprigs}
Given a pair of distinct sprigs $\hK, \hL \in \hcD$, let $U(\hK | \hL)$ be the complementary region of $\hK$ that contains $\hL$, and let $U(\hL | \hK)$ be the complementary region of $\hL$ that contains $\hK$. The intersection
\[ U(\hK, \hL) := U(\hK | \hL) \cap U(\hK | \hL) \]
is called the \emph{region between $\hK$ and $\hL$}.

\begin{lemma}\label{lemma:RegionBetweenSprigs}
	Let $\hK, \hL \in \hcD$ be distinct. Then $U(\hK, \hL)$ is path-connected. Moreover,
	\[ U(\hK, \hL) \cap S^1_u = I(\hK, \hL) \sqcup I(\hL, \hK), \]
	where $I(\hK, \hL)$ ($I(\hL, \hK)$) is the unique complementary interval of $\partial \hK \cup \partial \hL$ with initial endpoint in $\partial \hK$ (resp. $\partial \hL$) and terminal endpoint in $\partial \hL$ (resp. $\partial \hK$).
\end{lemma}
\begin{proof}
	Let $\hA = \hP \setminus U(\hK | \hL)$ and $\hB = \hP \setminus U(\hL | \hK)$, which are disjoint, nonseparating saturated continua. Then $\hA \cup \hB$ is nonseparating because the disc has the \emph{Phragmen-Brouwer property} \cite[\S~II.4]{Wilder}: If $\hA, \hB \subset \hP$ are disjoint and nonseparating, then $\hA \cup \hB$ is nonseparating. Then $U(\hK, \hL)$ is connected, and hence path-connected, since it can be written as $U(\hK, \hL) = \hP \setminus (\hA \cup \hB)$. 
	
	Note that $U(\hK | \hL) \cap S^1_u = (k, k')$ for $k, k' \in \partial \hK$, while $U(\hL | \hK) \cap S^1_u = (l', l)$ for $l', l \in \partial \hL$. See the left side of Figure~\ref{figure:SeparatingSprigs}. Thus $U(\hK, \hL) \cap S^1_u = U(\hK | \hL) \cap U(\hL | \hK) \cap S^1_u = (k, l) \sqcup (l', k')$. By Corollary~\ref{corollary:SprigsUnlinked}, $\partial \hK$ and $\partial \hL$ do not separate each other, so $(k, l)$ is the unique complementary interval of $\partial \hK \cup \partial \hL$ with initial point in $\partial \hK$ and terminal point in $\partial \hL$, and similarly for $(l', k')$.
\end{proof}

\begin{figure}[h]
	\labellist
	\tiny\hair 2pt
	
	\pinlabel $\textcolor{red}{\hK}$ [l] at 28 90
	\pinlabel $\textcolor{red}{k}$ [tr] at 21 48
	\pinlabel $\textcolor{red}{k'}$ [br] at 21 132
	
	\pinlabel $\textcolor{red}{\hL}$ [r] at 165 90
	\pinlabel $\textcolor{red}{l}$ [tl] at 159 48
	\pinlabel $\textcolor{red}{l'}$ [bl] at 159 132	
	
	\pinlabel ${U(\hK, \hL) = U(\hL, \hK)}$ [c] at 90 90
	
	\pinlabel ${I(\hK, \hL)}$ [t] at 87 7
	
	\pinlabel ${I(\hL, \hK)}$ [b] at 87 171
	
	\pinlabel $\textcolor{gray}{c}$ [l] at 295 140
	\pinlabel $\textcolor{red}{\mathbf{C}}$ [c] at 275 120
	\pinlabel $\textcolor{red}{\mathbf{F(U)}}$ [r] at 262 100
	\pinlabel $\textcolor{white}{U}$ [c] at 225 80
	
	\endlabellist
	\centering
	\includegraphics{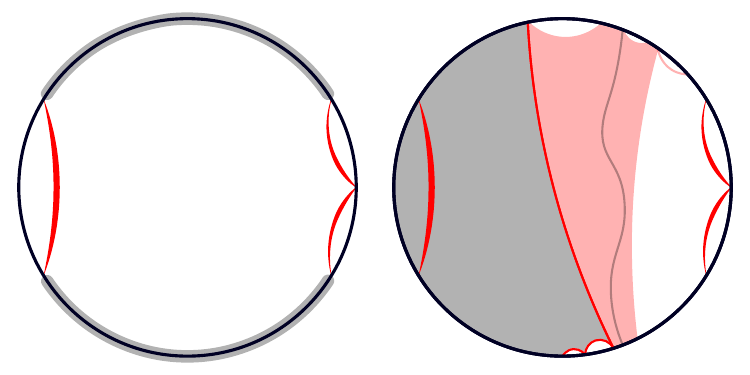}
	\caption{Separating sprigs.} \label{figure:SeparatingSprigs}
\end{figure}
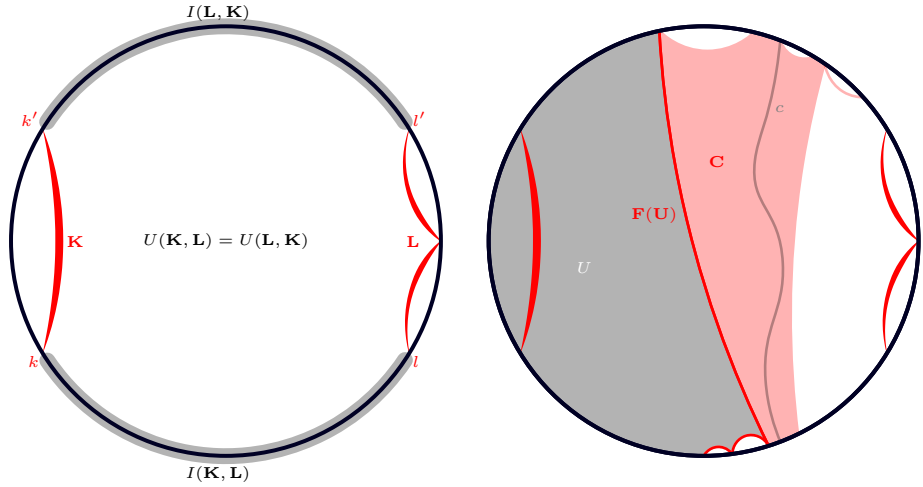

It follows that a sprig $\hK' \in \hcD$ is contained in $U(\hK, \hL)$ if and only if $\partial \hK' \subset I(\hK, \hL) \sqcup I(\hL, \hK)$. In addition, the preceding lemma, together with Corollary~\ref{corollary:SeparatingSprigsViaEnds}, implies the following.

\begin{lemma}\label{lemma:SeparatingViaS1Intervals}
	Let $\hK, \hL \in \hcD$ be distinct. Then $\hK' \in \hcD$ separates $\hK$ from $\hL$ if and only if $\partial \hK'$ intersects both $I(\hK, \hL)$ and $I(\hL, \hK)$.
\end{lemma}

\begin{proposition}\label{proposition:SprigSeparatingSprigs}
	Let $\hK, \hL \in \hcD$ be distinct. Then some $\hK' \in \hcD$ separates $\hK$ from $\hL$.
\end{proposition}
\begin{proof}
	See the right side of Figure~\ref{figure:SeparatingSprigs}. Since $U(\hK, \hL)$ is path-connected, we can find an arc $c \subset U(\hK, \hL)$ with initial point in $I(\hK, \hL)$ and terminal point in $I(\hL, \hK)$. Let $\hC = \hcD(c)$ be the saturation of $c$, which is contained in $U(\hK, \hL)$, and let $U$ be the complementary region of $\hC$ that contains $\hK$. Then the face $F(U)$ separates $\hK$ from $\hL$, since it has points in both $I(\hK, \hL)$ and $I(\hL, \hK)$, and $F(U)$ is contained in a sprig $\hK'$ by Lemma~\ref{lemma:FaceInSprig}.
\end{proof}

Note that this provides an alternative proof of Lemma~\ref{lemma:TwoEndedSprigs}: simply fix two sprigs and take any sprig that separates them.

\subsection{Separation intervals}
Let us study the collection of all sprigs that separate two fixed sprigs.

\begin{definition}
	Let $(\hK_\alpha, \hK_\beta)$ be an ordered pair of distinct sprigs. The corresponding \emph{separation interval} is the set
	\[ \sigma(\hK_\alpha, \hK_\beta) := \{ \hK \in \hcD \mid \hK \text{ separates } \hK_\alpha \text{ from } \hK_\beta \}, \]
	together with the binary relation $\prec$ defined by setting $\hK \prec \hK'$ whenever $\hK$ separates $\hK_\alpha$ from $\hK'$.
\end{definition}

We will show that this defines a linear order, and that the separation interval between any pair of distinct sprigs is order-isomorphic to the real line.

Fix an ordered pair of distinct sprigs $(\hK_\alpha, \hK_\beta)$. For brevity, we write $\sigma_{\alpha\beta} := \sigma(\hK_\alpha, \hK_\beta)$ for the corresponding separation interval, and $U_{\alpha\beta} := U(\hK_\alpha, \hK_\beta)$ for the region between $\hK_\alpha$ and $\hK_\beta$. By Lemma~\ref{lemma:RegionBetweenSprigs}, this is connected, and it intersects $S^1_u$ in the disjoint union of two intervals. We abbreviate and label these intervals by
\begin{align*}	I_{\alpha\beta} &:= I(\hK_\alpha, \hK_\beta) = (k_\alpha, k_\beta) \\
				I_{\beta\alpha} &:= I(\hK_\beta, \hK_\alpha) = (l_\beta, l_\alpha)	\end{align*}
where $k_\alpha, l_\alpha \in \partial \hK_\alpha$ and $l_\beta, k_\beta \in \partial \hK_\beta$. See the left side of Figure~\ref{figure:SeparationInterval}. 

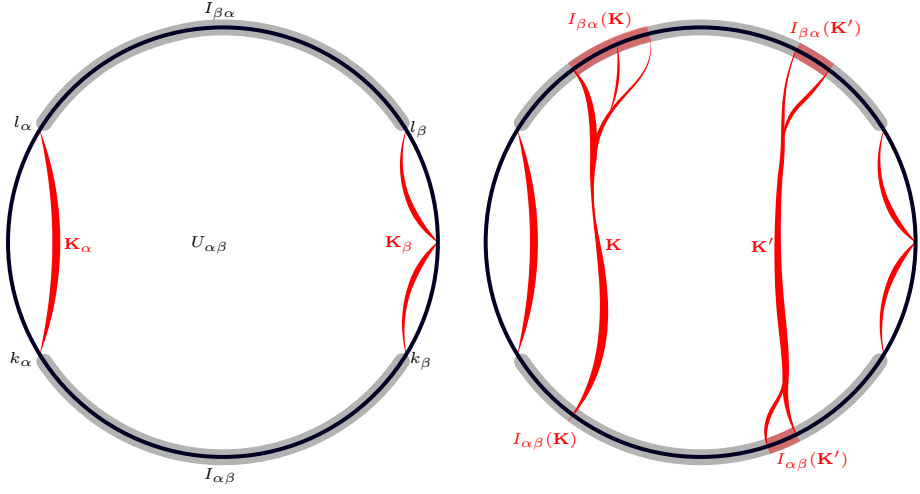
\begin{figure}[h]
	\labellist
	\tiny\hair 2pt
	
	\pinlabel $\textcolor{red}{\hK_\alpha}$ [l] at 28 89
	\pinlabel $\textcolor{red}{\hK_\beta}$ [r] at 164 89
	
	\pinlabel $k_\alpha$ [tr] at 20 50
	\pinlabel $k_\beta$ [tl] at 159 50
	\pinlabel $I_{\alpha\beta}$ [t] at 89 7
	
	\pinlabel $l_\beta$ [bl] at 159 128
	\pinlabel $l_\alpha$ [br] at 20 131
	\pinlabel $I_{\beta\alpha}$ [b] at 89 172
	
	\pinlabel $U_{\alpha\beta}$ [c] at 85 89
	
	\pinlabel $\textcolor{red}{\hK}$ [l] at 232 89
	\pinlabel $\textcolor{red}{I_{\beta\alpha}(\hK)}$ [b] at 232 168	
	\pinlabel ${\textcolor{red}{I_{\alpha\beta}(\hK)}}$ [tr] at 225 20
	
	\pinlabel $\textcolor{red}{\hK'}$ [r] at 300 89
	\pinlabel $\textcolor{red}{I_{\beta\alpha}(\hK')}$ [bl] at 302 165	
	\pinlabel ${\textcolor{red}{I_{\alpha\beta}(\hK')}}$ [tl] at 297 13

	\endlabellist
	\centering
	\includegraphics{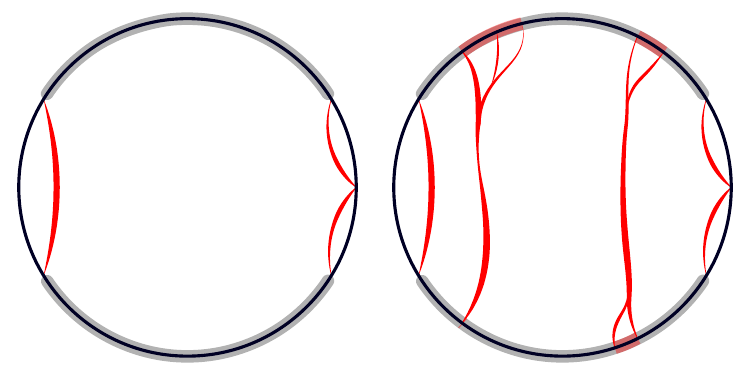}
	\caption{The separation interval.} \label{figure:SeparationInterval}
\end{figure}

For each $\hK \in \hcD$, define 
\begin{align*}
\partial_{\alpha\beta} \hK &:= \hK \cap I_{\alpha\beta}\\
\partial_{\beta\alpha} \hK &:= \hK \cap I_{\beta\alpha}.
\end{align*}
By Lemma~\ref{lemma:SeparatingViaS1Intervals}, $\hK \in \sigma_{\alpha\beta}$ if and only if $\partial_{\alpha\beta} \hK \neq \emptyset \neq \partial_{\beta\alpha} \hK$.

\begin{definition}
	Given subsets $X, Y$ of linearly ordered set $(Z, <)$, we take $X < Y$ to mean that $x < y$ for all $x \in X$ and $y \in Y$. We say that $X$ and $Y$ are \emph{comparable} if either $X < Y$ or $Y < X$.
\end{definition}
	
The intervals $I_{\alpha\beta}$ and $I_{\beta\alpha}$ come with linear orders induced by their orientations. For distinct $\hK, \hK' \in \sigma_{\alpha\beta}$, we would like to say that $\partial_{\alpha\beta} \hK$ and $\partial_{\alpha\beta} \hK'$ are comparable in $I_{\alpha\beta}$, and that $\partial_{\beta\alpha} \hK$ and $\partial_{\beta\alpha} \hK'$ are comparable in $I_{\beta\alpha}$.

For each $\hK \in \sigma_{\alpha\beta}$, let $I_{\alpha\beta}(\hK)$ be the minimal sub-interval of $I_{\alpha\beta}$ that contains $\partial_{\alpha\beta} \hK$, and let $I_{\beta\alpha}(\hK)$ be the minimal sub-interval of $I_{\beta\alpha}$ that contains $\partial_{\beta\alpha} \hK$. Then $\partial_{\alpha\beta} \hK$ and $\partial_{\alpha\beta} \hK'$ are comparable if and only if $I_{\alpha\beta}(\hK)$ and $I_{\alpha\beta}(\hK')$ are disjoint, and similarly for $\partial_{\beta\alpha} \hK$ and $\partial_{\beta\alpha} \hK'$. See the right side of Figure~\ref{figure:SeparationInterval}. 

\begin{lemma}
	If $\hK, \hK' \in \sigma_{\alpha\beta}$ are distinct, then 
	\[ I_{\alpha\beta}(\hK) \cap I_{\alpha\beta}(\hK') = \emptyset = I_{\beta\alpha}(\hK) \cap I_{\beta\alpha}(\hK'). \]
\end{lemma}
\begin{proof}
	If $I_{\alpha\beta}(\hK)$ intersects $I_{\alpha\beta}(\hK')$, then after possibly switching $\hK$ with $\hK'$, there must be some end of $\hK$ in the interior of $I_{\alpha\beta}(\hK')$. But the ends of $\hK$ are contained in a single complementary interval of $\hK'$ (Corollary~\ref{corollary:SprigsUnlinked}), so this means that \emph{all} of the ends of $\hK$ are contained in the interior of $I_{\alpha\beta}(\hK')$. This is impossible, since $\hK$ must also have an end in $I_{\beta\alpha}$. Thus we must have $I_{\alpha\beta}(\hK) \cap I_{\alpha\beta}(\hK') = \emptyset$, and a similar argument shows that $I_{\beta\alpha}(\hK) \cap I_{\beta\alpha}(\hK') = \emptyset$.
\end{proof}

Consequently, the linear orders on $I_{\alpha\beta}$ and $I_{\beta\alpha}$ induce linear orders on
\begin{align*}
							&\cI_{\alpha\beta} := \{ I_{\alpha\beta}(\hK) \mid \hK \in \sigma \} \\
\text{and}\hspace{.1in} 	&\cI_{\beta\alpha} := \{ I_{\beta\alpha}(\hK) \mid \hK \in \sigma \}.
\end{align*}
By the following lemma, we can think of $\hK \mapsto I_{\alpha\beta}(\hK)$ as an order-isomorphism $\sigma_{\alpha\beta} \to \cI_{\alpha\beta}$, and $\hK \mapsto I_{\beta\alpha}(\hK)$ as an anti-order-isomorphism $\sigma_{\alpha\beta} \to \cI_{\beta\alpha}$. 
 
\begin{lemma}\label{lemma:SeparationIntervalEquivalencies}
	The following are equivalent for all $\hK, \hK' \in \sigma$:
	\begin{enumerate}
		\item $I_{\alpha\beta}(\hK) < I_{\alpha\beta}(\hK')$,
		\item $I_{\beta\alpha}(\hK) > I_{\beta\alpha}(\hK')$,
		\item $\hK$ separates $\hK_\alpha$ from $\hK'$, and
		\item $\hK'$ separates $\hK$ from $\hK_\beta$.
	\end{enumerate}
\end{lemma}
\begin{proof}	
	The equivalence $(1) \Leftrightarrow (2)$ follows from Corollary~\ref{corollary:SprigsUnlinked}, while the equivalences $(1) + (2) \Leftrightarrow (3)$ and $(1) + (2) \Leftrightarrow (4)$ follow  from Corollary~\ref{corollary:SeparatingSprigsViaEnds}. 
\end{proof}

In particular:

\begin{corollary}
	Let $\hK_\alpha, \hK_\beta \in \hcD$ be distinct. Then the separation interval $\sigma(\hK_\alpha, \hK_\beta)$ is a linearly ordered set.
\end{corollary}

Let
\begin{align*}
					&\bigcup \cI_{\alpha\beta} := \bigcup_{\hK \in \sigma_{\alpha\beta}} I_{\alpha\beta}(\hK) \\
\text{and}\hspace{.1in} 	&\bigcup \cI_{\beta\alpha} := \bigcup_{\hK \in \sigma_{\alpha\beta}} I_{\beta\alpha}(\hK).
\end{align*}
We will use the following lemma to show that $\sigma_{\alpha\beta}$ is order-isomorphic to the real line. 

\begin{lemma}\label{lemma:IntervalsCover}
	$\bigcup \cI_{\alpha\beta} = I_{\alpha\beta}$ and $\bigcup \cI_{\beta\alpha} = I_{\beta\alpha}$.
\end{lemma}
\begin{proof}
	Let $\hA$ be the union of $\hK_\alpha$, $\hK_\beta$, and all sprigs in $\hcD$ that intersect $I_{\beta\alpha}$. This is a saturated continuum, since it is the saturation of the closed interval $\overline{I_{\beta\alpha}}$.
	
	Fix a point $s \in I_{\alpha\beta}$. If $s \in \hA$, then the sprig $\hcD(s)$ is contained in $\sigma_{\alpha\beta}$ since it has ends in both $I_{\alpha\beta}$ and $I_{\beta\alpha}$, and hence $s \in \bigcup \cI_{\alpha\beta}$.
	
	If $s \notin \hA$, then it is contained in $U \cap S^1_u$ for some complementary region $U$ of $\hA$. By Lemma~\ref{lemma:FaceInSprig}, the face $F(U)$ is contained in some sprig $\hK \subset \hA$. It follows that $\hK \in \sigma_{\alpha\beta}$, and $s \in I_{\alpha\beta}(\hK)$, so we have once again that $s \in \bigcup \cI_{\alpha\beta}$. Thus $\bigcup \cI_{\alpha\beta} = I_{\alpha\beta}$, and a similar argument shows that $\bigcup \cI_{\beta\alpha} = I_{\beta\alpha}$.
\end{proof}

\begin{proposition}\label{proposition:SeparationIntervalR}
	Let $\hK_\alpha, \hK_\beta \in \hcD$ be distinct. Then the separation interval $\sigma(\hK_\alpha, \hK_\beta)$ is order-isomorphic to the real line.
\end{proposition}
\begin{proof}
	The preceding lemma implies that $\cI_{\alpha\beta}$ is a separable and complete linearly ordered set with no maximum and minimum, which characterizes the linear order on $\bR$. Since $\sigma(\hK_\alpha, \hK_\beta)$ is order-isomorphic to $\cI_{\alpha\beta}$ it is also order-isomorphic to $\bR$.
\end{proof}

The following properties of separation intervals follow immediately from Lemma~\ref{lemma:SeparationIntervalEquivalencies}.

\begin{lemma}\label{lemma:SeparationIntervalProperties}
	Let $\hK_\alpha, \hK_\beta \in \hcD$ be distinct. 
	\begin{enumerate}
		\item Then $\hK \mapsto \hK$ defines an anti-order-isomorphism $\sigma(\hK_\alpha, \hK_\beta) \to \sigma(\hK_\beta, \hK_\alpha)$.
		\item Let $\hK_a, \hK_b \in \sigma(\hK_\alpha, \hK_\beta)$, with $\hK_a \prec \hK_b$. Then 
		\[ \sigma(\hK_a, \hK_b) = \{ \hK \in \sigma(\hK_\alpha, \hK_\beta) \mid \hK_a \prec \hK \prec \hK_b\},\]
		and the inclusion $\sigma(\hK_a, \hK_b) \hookrightarrow \sigma(\hK_\alpha, \hK_\beta)$ is order-preserving.
	\end{enumerate}
\end{lemma}

\section{Decompositions II: Linking}\label{section:Linking}
Now that we have some tools to work with the individual sprig decompositions, we can study the relationship between them. In fact, the results in this section apply to any spidery pair in the disc (Definition~\ref{definition:SpideryPair}).

Let $A, B \subset S^1$ be closed, disjoint, nontrivial subsets of the circle. We define the \emph{linking number} $\lk(A,B)$ to be the number of complementary intervals of $A$ that intersect $B$, which is finite and symmetric by the following lemma. We say that $A$ and $B$ are \emph{unlinked} if $\lk(A, B) = 1$, and \emph{linked} of $\lk(A, B) \geq 2$. The terminology comes from the fact that $A$ and $B$ are linked if and only if there are pairs $\{a, a'\} \subset A$ and $\{b, b'\} \subset B$ that are linked as $0$-spheres in $S^1$.

\begin{lemma}\label{lemma:LinkingNumber}
	Let $A, B \subset S^1$ be closed, disjoint, and nontrivial. Then $\lk(A, B)$ is finite and equal to $\lk(B, A)$.
\end{lemma}
\begin{proof}
	Fix an orientation on $S^1$. An \emph{$AB$-interstitial interval} is an oriented interval $(a, b) \subset S^1$ that is disjoint from $A \cup B$, with $a \in A$ and $b \in B$. It's easy to see that each complementary interval of $A$ that intersects $B$ contains a unique $AB$-interstitial interval, and each $AB$-interstitial interval is contained in a unique complementary interval of $A$ that intersects $B$. Similarly, each complementary interval of $B$ that intersects $A$ contains a unique $AB$-interstitial interval, and each $AB$-interstitial interval is contained in a unique complementary interval of $B$ that intersects $A$. Thus there is a bijective correspondence between the complementary intervals of $B$ that intersect $A$, the $AB$-interstitial intervals, and the complementary intervals of $B$ that intersect $A$, and hence $\lk(A, B) = \lk(B, A)$.
	
	To complete the lemma, it suffices to show that there are only finitely many $AB$-interstitial intervals. Otherwise, we would have an infinite sequence of distinct $AB$-interstitial intervals $(a_i, b_i)_{i = 1}^\infty$. These are pairwise disjoint intervals in the circle, so their diameters must go to zero, and we can assume after taking a subsequence that they converge to a single point $s \in S^1$. Then since $A$ and $B$ are closed, we have $s = \lim a_i \in A$ and $s = \lim b_i \in B$, which contradicts the assumption that $A$ and $B$ are disjoint.
\end{proof}

The following lemma is an immediate consequence of the usual pigeonhole principle; we will use it repeatedly.

\begin{lemma}[Linking Pigeonhole Principle]
	Let $A, B \subset S^1$ be closed, disjoint, and nontrivial, and let $X \subset S^1$ be a subset with $A \subset X$. If $X$ has fewer than $\lk(A, B)$ connected components, then $X \cap B \neq \emptyset$.
\end{lemma}

In particular, if $A, B \subset S^1$ are linked, then any connected set that contains $A$ must intersect $B$.

\subsection{The linked region}\label{subsection:LinkedRegion}
Let $\hK \in \hcD^+$ and $\hL \in \hcD^-$ be positive and negative sprigs. If $\partial \hK \cap \partial \hL = \emptyset$, then we define their linking number by 
\[ \lk(\hK, \hL) := \lk(\partial \hK, \partial \hL). \]
We say that $\hK$ and $\hL$ are \emph{unlinked} when $\lk(\hK, \hL) = 1$ and \emph{linked} when $\lk(\hK, \hL) \geq 2$. If $\partial \hK \cap \partial \hL \neq \emptyset$, then their linking number is undefined, and they are neither linked nor unlinked. One could also define linking numbers for pairs of sprigs of the same kind, but these are always unlinked by Corollary~\ref{corollary:SprigsUnlinked}.

The sprigs through a point $p \in P$ always have disjoint ends (Corollary~\ref{corollary:DisjointEnds}), so
\[ \lk(p) := \lk(\hcD^+(p), \hcD^-(p)) \]
is always defined, and we call $p$ a \emph{linked point} or \emph{unlinked point} accordingly. The set of all linked points is called the \emph{linked region} and denoted by $P_l \subset P$. We will show that $P_l$ is closed in $P$ and nontrivial.

\subsection{Linked sequences}
The following lemma will be useful when considering sequences of linked points.

\begin{lemma}\label{lemma:LinkedSequencesComplementaryRegions}
	Let $(p_i)_{i = 1}^\infty$ be a sequence of points in $P_l$ that converge to a point $p \in P$. Then $\{p_i\}$ intersects at most finitely many complementary regions of $\hcD^+(p)$ and of $\hcD^-(p)$.
\end{lemma}
\begin{proof}
	Let $\hK_i = \hcD^+(p_i)$ and $\hL_i = \hcD^-(p_i)$ be the positive and negative sprigs through each $p_i$, and let $\hK = \hcD^+(p)$ and $\hL = \hcD^-(p)$ be the positive and negative sprigs through $p$. We have $\varlimsup \partial \hK_i \subset \partial \hK$ and $\varlimsup \partial \hL_i \subset \partial \hL$ by upper semicontinuity. Since $p \in P$, Corollary~\ref{corollary:DisjointEnds} implies that $\partial \hK$ is disjoint from $\partial \hL$, so $\varlimsup \partial \hK_i$ is disjoint from $\varlimsup \partial \hL_i$.

	Suppose that $\{p_i\}$ intersects infinitely many complementary regions of $\hK$. Then after taking a subsequence we can assume that each $p_i$ is contained in a distinct complementary region $U_i$. Then $\hK_i \subset U_i$ and $\partial \hK_i \subset U_i \cap S^1_u =: J_i$ for each $i$.
		
	For each $i$, let $k_i$ be an end of $\hK_i$, which is contained in $J_i$. By the Linking Pigeonhole Principle, we can also find an end $l_i$ of $\hL_i$ that lies in $J_i$. The $J_i$ are pairwise disjoint intervals in the circle, so their diameters must go to zero, and we can assume after taking a subsequence that they converge to a single point $s \in S^1_u$. Then $\lim k_i = s = \lim l_i$, which contradicts the fact that $\varlimsup \partial \hK_i$ is disjoint from $\varlimsup \partial \hL_i$. Thus the $p_i$ can visit only finitely many complementary regions of $\hK$, and the same argument shows that they can visit only finitely many complementary regions of $\hL$.
\end{proof}
	
\begin{lemma}\label{lemma:LinkedSequencesN}
	Let $(p_i)_{i = 1}^\infty$ be a sequence of points in $P_l$ that converge to a point $p \in P$, and suppose that there is an integer $n$ such that $\lk(p_i) \geq n$ for all $i$. Then $\lk(p) \geq n$.
\end{lemma}
\begin{proof}
	We will continue to use the notation and observations in the first paragraph of the preceding proof.
	
	For each $i$, and choose ends $k_i^0, k_i^1, \cdots, k_i^n \in \partial \hK_i$ and $l_i^0, l_i^1, \cdots, l_i^n \in \partial \hL_i$ such that $k_i^0, l_i^0, k_i^1, l_i^1, \cdots, k_i^n, l_i^n$ is positively ordered in $S^1_u$. After taking a subsequence, we can assume that $\lim_{i \to \infty} k_i^j$ converges to an end $k^j \in \partial \hK$ for each $j$ and $\lim_{i \to \infty} l_i^j$ converges to an end $l^j \in \partial \hL$ for each $j$.
	
	It suffices to show that $k^0, l^0, k^1, l^1, \cdots, k^n, l^n$ are pairwise distinct and positively ordered in $S^1_u$. Since $\partial \hK$ is disjoint from $\partial \hL$, it follows that $k^j \neq l^{j'}$ for all $j, j'$. Suppose that $k^j = k^{j'}$ for some $j \neq j'$. Then either $\lim_{i \to \infty} (k_i^j, k_i^{j'}) = k^j$ or $\lim_{i \to \infty} (k_i^{j'}, k_i^j) = k^j$. Either case is impossible, since each of these intervals contain at least one of the $l^{j''}$'s, which cannot accumulate on $k^j \in \partial \hK$. Therefore, $k^j \neq k^{j'}$ for all $j \neq j'$ and a similar argument shows that $l^j \neq l^{j'}$ for all $j \neq j'$.
	
	The fact that $k_i^0, l_i^0, k_i^1, l_i^1, \cdots, k_i^n, l_i^n$ is positively ordered for each $i$ easily implies that $k^0, l^0, k^1, l^1, \cdots, k^n, l^n$ is positively ordered.
\end{proof}

\begin{proposition}\label{proposition:PlClosed}
	$P_l$ is closed in $P$.
\end{proposition}
\begin{proof}
	Let $(p_i)$ be a sequence of points in $P_l$ that converge to $p \in P$. Then $\lk(p_i) \geq 2$ for all $i$. By the preceding lemma, $\lk(p) \geq 2$ which means that $p \in P_l$.
\end{proof}

The following lemma will be used in the proof of the Homotopy Closing Lemma.

\begin{lemma}\label{lemma:LinkedSequences3}
	Let $(p_i)_{i = 1}^\infty$ be a sequence of points in $P$ that converge to a point $p \in P$, and suppose that $\lk(p_i) \geq 3$ for all $i$. Then $p_i \in \hcD^+(p) \cap \hcD^-(p)$ for all sufficiently large $i$.
\end{lemma}
\begin{proof}
	With the notation above, suppose that there are infinitely many $i$ for which $p_i \notin \hK$. By Lemma~\ref{lemma:LinkedSequencesComplementaryRegions} we can assume after passing to a subsequence that each $\hK_i$ is contained in the same complementary region $U$ of $\hK$.
	
	Then the ends of each $\hK_i$ are contained in the corresponding complementary interval, which can be written in the form $U \cap S^1_u = (k, k')$ for $k, k' \in \partial \hK$. Thus $\limsup \partial \hK_i \subset \{k, k'\}$, so we can find sequences of points $s_i, s'_i \in (k, k')$ such that $\partial \hK_i (k, s_i) \sqcup (s'_i, k')$ for each $i$, and $\lim s_i = k$ and $\lim s'_i = k'$. Since $\hK_i$ and $\hL_i$ have linking number at least $3$, the Linking Pigeonhole Principle implies that some end $l_i \in \partial \hL_i$ is contained in either $(k, s_i)$ or $(s'_i, k'_i)$. Then $\varlimsup l_i$ contains either $k$ or $k'$, contradicting the fact that $\varlimsup \partial \hK_i$ is disjoint from $\varlimsup \partial \hL_i$. Thus $p_i \in \hK$ for all but finitely many $i$, and a similar argument shows that $p_i \in \hL$ for all but finitely many $i$.
\end{proof}

\subsection{Nontriviality}
\begin{proposition}\label{proposition:PlNontrivial}
	$P_l \neq \emptyset$.
\end{proposition}
\begin{proof}
	By Lemma~\ref{lemma:TwoEndedSprigs}, we can find a positive sprig $\hK$ with two distinct ends $k_\alpha$ and $k_\beta$. By Lemma~\ref{lemma:NoBigons}, the negative sprigs $\hN_\alpha, \hN_\beta$ through these points are distinct, so $\sigma(\hN_\alpha, \hN_\beta)$ is nontrivial. Each $\hN \in \sigma(\hN_\alpha, \hN_\beta)$ separates $k_\alpha$ from $k_\beta$, but $\hN$ need not be linked with $\hK$ since we may have $\partial \hN \cap \partial \hK \neq \emptyset$. However, we can use the following lemma to find an $\hN \in \sigma(\hN_\alpha, \hN_\beta)$ that intersects $\chK$. Then $\partial \hN \cap \partial \hK = \emptyset$ (Corollary~\ref{corollary:DisjointEnds}), so $\hN$ is linked with $\hK$.
\end{proof}

\begin{lemma}\label{lemma:DenseInSeparationInterval}
	Let $k_\alpha$ and $k_\beta$ be distinct ends of a positive sprig $\hK$, and let $\hN_\alpha = \hcD^+(k_\alpha)$ and $\hN_\beta = \hcD^+(k_\beta)$. Then the separation interval $\sigma(\hN_\alpha, \hN_\beta)$ contains a dense set of sprigs that intersect $\chK$.
\end{lemma}
\begin{proof}
	See Figure~\ref{figure:DenseInSeparationInterval}. Given a pair of sprigs $\hN_a \prec \hN_b \in \sigma(\hN_\alpha, \hN_\beta)$, we must find a sprig $\hN \in \sigma(\hN_a, \hN_b)$ that intersects $\chK$ nontrivially. To see this, we will show that there are sprigs $\hN'_a \prec \hN'_b \in \sigma(\hN_a, \hN_b)$ such that $I(\hN'_a, \hN'_b)$ is disjoint from $\partial \hK$. By the same argument, we can find sprigs $\hN''_a \prec \hN''_b \in \sigma(\hN'_a, \hN'_b)$ such that $I(\hN''_b, \hN''_a)$ is disjoint from $\partial \hK$;  note that $I(\hN''_a, \hN''_b) \subset I(\hN'_a, \hN'_b)$ is already disjoint from $\partial \hK$. Then it suffices to take any $\hN \in \sigma(\hN''_a, \hN''_b)$: such a sprig must intersect $\hK$ because it separates $k_\alpha$ from $k_\beta$, and the intersection is in $\chK$ as opposed to $\partial \hK$ because the ends of $\hN$ are contained in $I(\hN''_a, \hN''_b) \cup I(\hN''_b, \hN''_a)$.
	
	\begin{figure}[h]
		\labellist
		\small\hair 2pt
		
		\pinlabel $\textcolor{red}{k_\alpha}$ [r] at 10 90
		\pinlabel $\textcolor{blue}{\hN_a}$ [bl] at 29 62
		\pinlabel $\textcolor{blue}{\hN'_a}$ [l] at 66 75
		\pinlabel $\textcolor{blue}{\hN'_b}$ [r] at 108 115
		\pinlabel $\textcolor{blue}{\hN_b}$ [r] at 143 128
		\pinlabel $\textcolor{red}{k_\beta}$ [l] at 170 89
		
		\pinlabel ${I(\hN_\alpha, \hN_\beta)}$ [tr] at 38 22
		\pinlabel $\textcolor{blue}{I(\hN'_a, \hN'_b)}$ [t] at 91 6
		
		\pinlabel ${I(\hN_\beta, \hN_\alpha)}$ [bl] at 139 158
		\pinlabel $\textcolor{blue}{I(\hN'_b, \hN'_a)}$ [b] at 84 174

		\pinlabel $\textcolor{blue}{\hN''_a}$ [l] at 255 75
		\pinlabel $\textcolor{blue}{\hN''_b}$ [r] at 280 115
		
		\pinlabel $\textcolor{blue}{I(\hN''_a, \hN''_b)}$ [t] at 268 6
		\pinlabel $\textcolor{blue}{I(\hN''_b, \hN''_a)}$ [b] at 265 174
		
		\endlabellist
		\centering
		\includegraphics{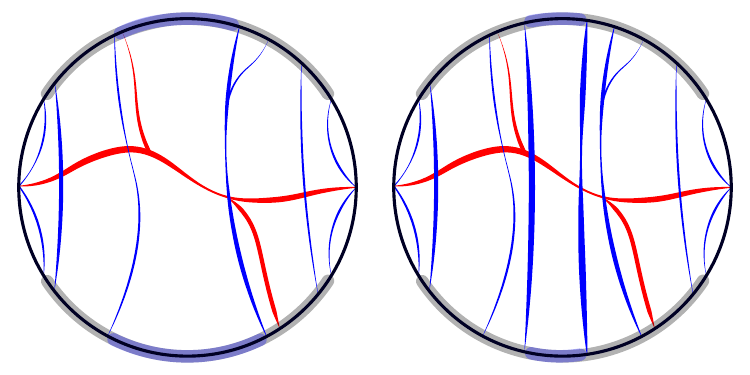}
		\caption{A dense set of sprigs intersecting the bounded part of $\hK$.} \label{figure:DenseInSeparationInterval}
	\end{figure}
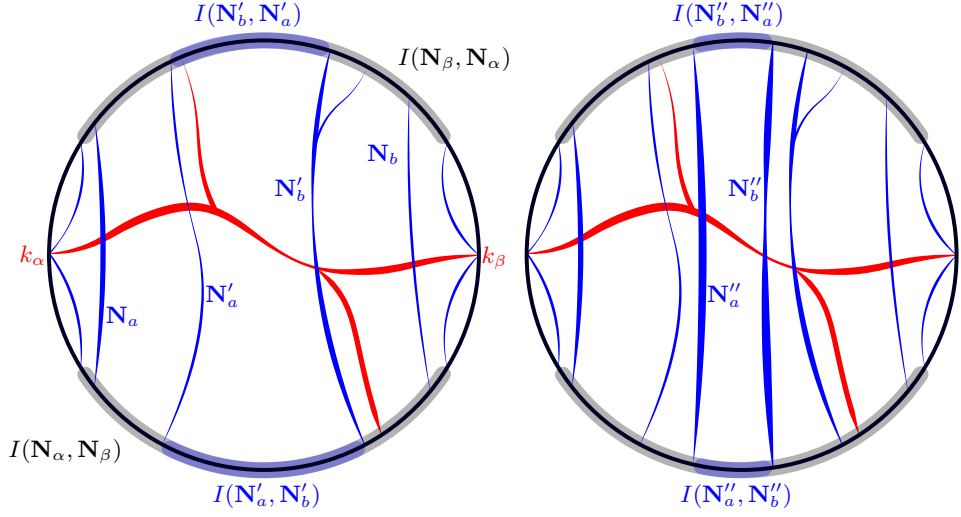
	
	Fix $\hN_a \prec \hN_b \in \sigma(\hN_\alpha, \hN_\beta)$, and choose a point $x \in I_{ab} \setminus \partial \hK$, where we are using the abbreviation $I_{ab} := I(\hN_a, \hN_b)$. By Lemma~\ref{lemma:IntervalsCover}, $x \in I_{ab}(\hN')$ for some sprig $\hN' \in \sigma(\hN_a, \hN_b)$. Note that the initial and terminal points of $\partial_{ab} \hN'$, which we will denote by $\partial_{ab}^\mp \hN'$, cannot both be contained in $\partial \hK$. Indeed, if $I_{ab}(\hN')$ is a point, then $\partial_{ab}^- \hN' = \partial_{ab}^+ \hN' = x$, which was chosen outside of $\partial \hK$; if it is an interval, then $\partial_{ab}^- \hN'$ and $\partial_{ab}^+ \hN'$ are distinct ends of $\hN'$, which can share at most one end with $\hK$.
	
	If $\partial_{ab}^+ \hN' \notin \partial \hK$, then we take $\hN'_a$ to be $\hN'$, and $\hN'_b$ to be any sprig in $\sigma(\hN'_a, \hN_b)$ that is close to $\hN'_a$, in the sense that $\partial \hK$ does not separate $\partial_{ab}^+ \hN'_a$ from $\partial_{ab}^- \hN'_b \}$. Then no end of $\hK$ lies in $I(\hN'_a, \hN'_b)$ as desired. If $\partial_{ab}^+(\hN') \in \partial \hK$, then $\partial_{ab}^- \hN \notin \partial \hK$, and we take $\hN'_b$ to be $\hN'$, and $\hN'_a$ to be any sprig in $\sigma(\hN_a, \hN'_b)$ that is close to $\hN'_b$.
	
	As noted above, we can complete the proof by repeating this argument.
\end{proof}

Although the proposition does not need the full power of this lemma, we will need it for the Homotopy Closing Lemma.

\section{Coarse transverse hyperbolicity}\label{section:CoarseHyperbolicity}
In this section, we will see that the coarsely hyperbolic behavior of our flow is reflected in the action of an $\omega$- or $\alpha$-sequence on the flowspace.

\subsection{Straightening flowlines}
Given a flowline $x \cdot \bR$, let $(x \cdot \bR)^g$ denote the corresponding oriented geodesic, running from $E^-(x)$ to $E^+(x)$. The nearest-point projection restricts to a proper surjective map
\[ \rho_{x \cdot \bR}: x \cdot \bR \to (x \cdot \bR)^g \]
that moves each point a uniformly bounded distance, independent of $x$. To see this, recall that $x \cdot \bR$ is contained in the $C$-neighborhood of $(x \cdot \bR)^g$, for a uniform constant $C$, which we can picture as a ``banana'' foliated by radius-$C$ hyperbolic discs. The nearest-point projection takes each point in $x \cdot \bR$ to the center of the corresponding disc, so $d(y, \rho_{x \cdot \bR}(y)) < C$ for all $y \in x \cdot \bR  $.

Since the endpoints of flowlines vary continuously, we can define a continuous, $\pi_1(M)$-equivariant \emph{straightening map}
\begin{align*}
	G: \tM 	&\to 		\tM					\\
	x	&\mapsto	\rho_{x \cdot \bR}(x)
\end{align*}
that takes each flowline onto its corresponding geodesic, while moving each point by a distance of at most $C$.

\begin{lemma}
	There is a constant $D > 0$ such that:
	\begin{enumerate}
		\item If $x_0, x_1 \in \tM$ are contained in the same weak positive leaf then there are times $t_0, t_1 \in \bR$ such that $x_0 \cdot [t_0, \infty)$ and $x_1 \cdot [t_1, \infty)$ have Hausdorff distance at most $D$.
		
		\item If $y_0, y_1 \in \tM$ are contained in the same weak negative leaf then there are times $s_0, s_1 \in \bR$ such that $y_0 \cdot (-\infty, s_0]$ and $y_1 \cdot (-\infty, s_1]$ have Hausdorff distance at most $D$.
	\end{enumerate}
\end{lemma}
\begin{proof}
	Let $D = 2C + \epsilon$ for an arbitrary constant $\epsilon > 0$. We will prove (1), and (2) follows from a similar argument.
	
	Since $x_0$ and $x_1$ are contained in the same weak positive leaf, the geodesics $G(x_0 \cdot \bR)$ and $G(x_1 \cdot \bR)$ have the same positive endpoint $z = E^+(x_0) = E^+(x_1)$. Therefore, we can find a horosphere $S$ centered at $z$ such that the distance between $G(x_0 \cdot \bR) \cap S'$ and $G(x_1 \cdot \bR) \cap S'$ is less than $\epsilon$ for every horosphere $S'$ centered at $z$ that lies forward of $S$. Since $G$ moves each point by a distance of at most $C$, it suffices to take $t_i = \sup \{ t \in \bR \mid G(x_i \cdot t) \in S\}$.
\end{proof}

Let $\U \tM$ be the unit tangent bundle of $\tM$, thought of as the space of pairs $(\gamma, y)$ consisting of an oriented geodesic $\gamma \subset \tM$ together with a point $y \in \gamma$. This comes with a natural flow, the \emph{geodesic flow} $\tTheta$, which takes a vector $(\gamma, y)$, after time $\tau \in \bR$, to the vector $(\gamma, y')$ such that $y'$ lies at signed distance $\tau$ from $y$.

Our straightening map $G$ has a natural lift
\begin{align*}
	F: \tM &\to \U \tM \\
	x &\mapsto ((x \cdot \bR)^g, G(x)),
\end{align*}
which takes each orbit of $\tPhi$ onto an orbit of $\tTheta$. In particular, if we fix $x \in \tM$, then for each time $t \in \bR$, there is a time $\tau \in \bR$ such that $F(x \cdot t) = \tTheta_{\tau(t)}F(x)$. Since $G$ moves each point a uniformly bounded distance, $\tau \to \pm\infty$ as $t \to \pm\infty$.

Each horosphere $S \subset \tM$ has two natural lifts to the unit tangent bundle, a \emph{stable horosphere} $S^+$ consisting of inward-pointing normal vectors and an \emph{unstable horosphere} $S^-$ consisting of outward-pointing normal vectors. If $S$ is centered at $z \in S^2_\infty$, then
\begin{align*}
							&S^+ = \{ (\gamma, y) \mid E^+(\gamma) = z, y \in S \}, \\
	\text{and}\hspace{.1in}	&S^- = \{ (\gamma, y) \mid E^-(\gamma) = z, y \in S \}.
\end{align*}
Here, we are abusing notation and writing $E^\pm(\gamma)$ for the positive/negative endpoint of $\gamma$. We say that $S^+$ \emph{points towards} $z$, and $S^-$ \emph{points away from} $z$. The geodesic flow is an Anosov flow whose strong stable and unstable leaves are exactly the stable and unstable horospheres \cite{Anosov}.

For each $x \in \tM$, the point $F(x)$ is contained in a unique stable horosphere $S^+(x)$, which points towards $E^+(x)$. Since $F$ takes the flowline through $x$ surjectively onto the corresponding lifted geodesic, every stable horosphere that points towards $E^+(x)$ is of the form $S^+(x \cdot t)$ for some $t$. Similar observations hold for unstable horospheres.

\subsection{Coarse contraction for $\omega$-sequences}\label{subsection:CoarseContraction}
We will use the map $F$, together with the Anosov behavior of $\tTheta$, to understand the way that $\omega$-sequences act on positive sprigs.

For each $q \in \hP$, we define
\[ \widehat{\Omega}(q) := \{ q' \in \hP \mid \he^+(q') = \he^+(q) \text{ and } \he^-(q') = \he^-(q) \}. \]

\begin{lemma}
	If $q \in P$, then $\widehat{\Omega}(q)$ is a compact subset of $P$.
\end{lemma}
\begin{proof}
	Let $z = \he^+(q)$ and $w = \he^-(q)$, and note that $z \neq w$. Then
	\begin{align*}
		\widehat{\Omega}(q) &= (\he^+)^{-1}(z) \cap (\he^-)^{-1}(w) \\
							&= (e^+)^{-1}(z) \cap (e^-)^{-1}(w)
	\end{align*}
	where the latter equality comes from the fact that $\he^+$ agrees with $\he^-$ on $S^1_u$. Thus $\widehat{\Omega}(q) \subset P$, and compactness follows from Lemma~\ref{lemma:BoundedSubsetsOfP}.
\end{proof}

\begin{proposition}[Coarse contraction]\label{proposition:CoarseContraction}
	Let $(g_i)$ be an $\omega$-sequence for $q \in \omega(p)$. Then
	\[ \varlimsup_{i \to \infty} g_i(p') \subset \widehat{\Omega}(q) \]
	for every $p' \in P$ with $e^+(p') = e^+(p)$.
\end{proposition}
\begin{proof}
	It suffices to show that $e^\pm(g_i(p'))$ converges to $e^\pm(q)$. Equivalently, we will show that the geodesics $F(\flow{g_i(p)})$ converge to $F(\flow{q})$.
	
	By the definition of an $\omega$-sequence, we can find a sequence of points $x_0, x_1, \cdots \in \flow{p}$ such that $\lim x_i = e^+(p)$ and $\lim g_i(x_i) = y \in \flow{q}$. For each $i$, let $S^+_i := S^+(x_i)$ be the stable horosphere that contains $F(x_i)$. Since $e^+(p) = e^+(p')$, each $S^+_i$ also contains $F(x'_i)$ for some point $x'_i \in \flow{p'}$. See Figure~\ref{figure:CoarseContraction}.
	
	\begin{figure}[h]
		\vspace{.2in}
		\labellist
		\tiny\hair 2pt
		\pinlabel ${S^+_{i-2}}$ 	[r] at 21 69
		\pinlabel ${S^+_{i-1}}$ 	[r] at 21 108
		\pinlabel ${S^+_i}$		[r] at 21 150
		
		\pinlabel $\textcolor{navy}{F(x_{i-2})}$ [l] at 81 42
		\pinlabel $\textcolor{navy}{F(x_{i-1})}$ [l] at 81 78
		\pinlabel $\textcolor{navy}{F(x_{i})}$ [l] at 81 121
		
		\pinlabel $\textcolor{navy}{F(x'_{i-2})}$ [r] at 48 57
		\pinlabel $\textcolor{navy}{F(x'_{i-1})}$ [r] at 62 87
		\pinlabel $\textcolor{navy}{F(x'_{i})}$ [r] at 72 125
		
		\pinlabel $\textcolor{navy}{\flow{p}^g}$ [t] at 80 8
		\pinlabel $\textcolor{navy}{\flow{p'}^g}$ [t] at 8 8
		
		\pinlabel $g_i$ [b] at 160 135
		
		\pinlabel $\textcolor{dkgreen}{\flow{q}^g}$ [l] at 275 142
		\pinlabel $\textcolor{navy}{g_i(\flow{p'}^g)}$ [tr] at 254 8
		\pinlabel $\textcolor{navy}{g_i(\flow{p}^g)}$ [tl] at 281 8
		
		\pinlabel $\textcolor{dkgreen}{F(y)}$ [tr] at 277 57
		\pinlabel $\textcolor{navy}{g_i(F(x_i))}$ [l] at 284 55
		\pinlabel $\textcolor{navy}{g_i(F(x'_i))}$ [br] at 266 63
		
		\pinlabel $S^+(y)$ [bl] at 304 99
		\pinlabel ${g_i(S^+_i)}$ [bl] at 325 10
		\endlabellist
		\centering
		\includegraphics{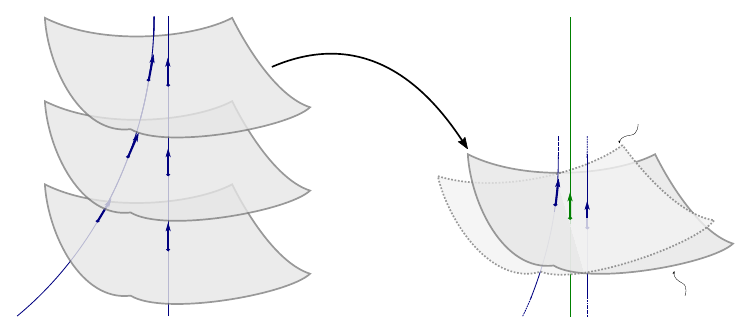}
		\caption{Coarse contraction.} \label{figure:CoarseContraction}
	\end{figure}
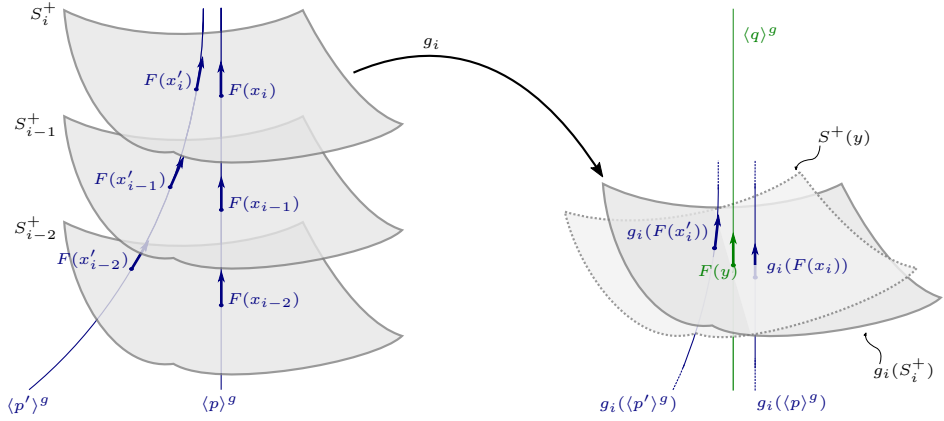
	
	For each $i$, $F(x_i)$ and $F(x'_i)$ are obtained by flowing $F(x_0)$ and $F(x'_0)$ for some time $\tau_i$ under the geodesic flow, where $\tau_i \to \infty$. The geodesic flow exponentially contracts stable horospheres, so the distance between $F(x_i)$ and $F(x'_i)$ goes to $0$ as $i \to \infty$. Each $g_i$ acts as an isometry on $\U \tM$, so we have $\lim F(g_i(x'_i)) = \lim F(g_i(x_i)) = F(y)$, and hence $\lim F(\flow{g_i(p')}) = F(\flow{q})$ as desired.
\end{proof}

For each $q \in \hP$, define
\[ \Omega(q) := \widehat{\Omega}(q) \cap \chcD^+(q), \]
which is a compact subset of $\chcD^+(q)$.

\begin{proposition}[Coarse contraction for sprigs]\label{proposition:CoarseContractionSprigs}
	Let $(g_i)$ be an $\omega$-sequence for $q \in \omega(p)$. Then
	\[ \varlimsup_{i \to \infty} g_i(p') \subset \Omega(q) \]
	for every $p' \in \chcD^+(p)$ 
\end{proposition}
\begin{proof}
	This follows from the preceding proposition, together with the fact that $\varlimsup \hcD^+(g_i(p)) \subset \hcD^+(q)$.
\end{proof}

This is the final ingredient in our proof of the Homotopy Closing Lemma. The corresponding result for $\alpha$-sequences is obtained by switching $+$ and $-$ in the proposition and in the definition of $\Omega(q)$.

\section{Closed orbits for quasigeodesic flows}\label{section:ClosedOrbits}
We turn to our main results.

\begin{recurrent_links_lemma}\label{lemma:RecurrentLinks}
	The linked region $P_l$ is closed, nontrivial, $\pi_1(M)$-invariant, and contains an $\omega$-recurrent point.
\end{recurrent_links_lemma}
\begin{proof}
	$P_l$ is closed by Proposition~\ref{proposition:PlClosed} and nontrivial by Proposition~\ref{proposition:PlNontrivial}. It is obviously $\pi_1(M)$-invariant, so it corresponds to a closed, flow-invariant subset $M_l = \pi(\flow{P_l}) \subset M$. Since $M$ is compact, $M_l$ is compact, so it must contain some minimal set. A minimal set is the closure of an almost-periodic orbit (see, e.g., \cite[Theorem~1.7]{Auslander}), which is \emph{a fortiori} $\omega$-recurrent.
\end{proof}

This, together with the Homotopy Closing Lemma, implies the Closed Orbits Theorem.

\subsection{Proof of the Homotopy Closing Lemma}
Fix an $\omega$-recurrent point $p \in P_l$, and a corresponding $\omega$-sequence $(g_i)_{i = 1}^\infty$. We must show that $g_i$ eventually fixes some point in $P$. By Lemma~\ref{lemma:ElementPreservingLeaf}, it suffices to show that $g_i$ eventually fixes some nontrivial sprig.

Let $\hK = \hcD^+(p)$ and $\hL = \hcD^-(p)$ be the positive and negative sprigs through $p$. By Corollary~\ref{corollary:DisjointEnds}, $\partial \hK \cap \partial \hL = \emptyset$. Since $\lim g_i(p) = p$, upper semicontinuity implies that $\varlimsup g_i(\hK) \subset \hK$ and $\varlimsup g_i(\hL) \subset \hL$.

If $\lk(p) \geq 3$, then Lemma~\ref{lemma:LinkedSequences3} implies that $g_i(\hK) = \hK$ for all $i$ sufficiently large. This completes the proof for the $(\geq 3)$-linked case, so we will assume from now on that $\lk(p) = 2$.

\subsubsection{A single complementary region}
After deleting any elements that fix $\hK$, we can assume that each $g_i$ takes $\hK$ into one of its complementary regions. Since $\lk(p) = 2$, there are exactly two complementary regions that contain ends of $\hL$.

\begin{claim}
	All but finitely many of the $g_i$ take $\hK$ into one of these two complementary regions.
\end{claim}

Let $W$ be a complementary region of $\hK$, and suppose that there is an infinite subsequence $(h_i) \subset (g_i)$ such that $h_i(\hK) \subset W$ for all $i$. Then $\partial h_i(\hK) \subset W \cap S^1_u$ for all $i$, so the Linking Pigeonhole Principle says that $W \cap S^1_u$ also contains ends $l_i \in \partial h_i(\hL)$. Let $l \in \overline{W \cap S^1_u}$ be an accumulation point of the $l_i$. This is an end of $\hL$, and $l \in W \cap S^1_u$ since the endpoints of this interval are ends of $\hK$. The $g_i(\hK)$ can visit only finitely many complementary regions (Lemma~\ref{lemma:LinkedSequencesComplementaryRegions}), so this proves the claim.

Consequently, it suffices to prove the Homotopy Closing Lemma with the additional assumption that $g_i(\hK) \subset W$ for all $i$, where $W$ is one of the two complementary regions that contain ends of $\hL$. The corresponding complementary interval is of the form $W \cap S^1_u = (k_\alpha, k_\beta)$ for ends $k_\alpha, k_\beta \in \partial \hK$. See Figure~\ref{figure:HomotopyClosing-U}.

\begin{figure}[h]
	\labellist
	\small\hair 2pt
	
	\pinlabel $p$ [tr] at 76 89
	\pinlabel $\textcolor{blue}{\hL}$ [b] at 42 90
	\pinlabel $\textcolor{red}{\hK}$ [r] at 80 125
	\pinlabel $\textcolor{white}{W}$  at 115 115
	\pinlabel $k_\alpha$ [t] at 90 8
	\pinlabel $k_\beta$ [b] at 90 172
	
	\endlabellist
	\centering
	\includegraphics{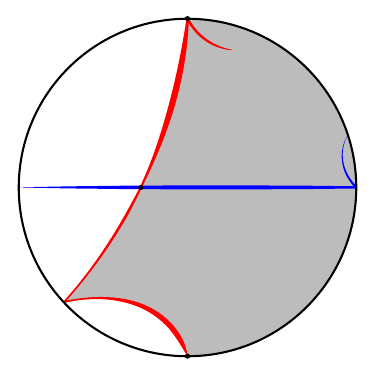}
	\caption{Homotopy closing for $2$-linked points.} \label{figure:HomotopyClosing-U}
\end{figure}
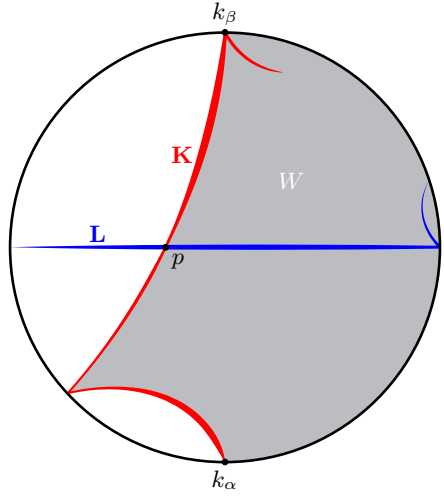

Let $\hN_\alpha$ and $\hN_\beta$ be the negative sprigs through $k_\alpha$ and $k_\beta$. These are distinct by Lemma~\ref{lemma:NoBigons}, so the corresponding separation interval $\sigma_{\alpha\beta} = \sigma(\hN_\alpha, \hN_\beta)$ is nontrivial and order-isomorphic to $\bR$ (Proposition~\ref{proposition:SeparationIntervalR}). We will show that $g_i$ eventually acts as a contraction on some sub-interval of $\sigma_{\alpha\beta}$ which means that it fixes some sprig in this interval.

\subsubsection{Translation- and rotation-like elements}
To begin, we will divide the elements of our $\omega$-sequence into two classes: We call $g_i$ \emph{translation-like} if $g_i(k_\alpha, k_\beta) \subset (k_\alpha, k_\beta)$, and \emph{rotation-like} if $g_i(k_\beta, k_\alpha) \subset (k_\alpha, k_\beta)$. See Figure~\ref{figure:HomotopyClosing-TranslationRotationLike}.

\begin{figure}[h]
	\labellist
	\small\hair 2pt
	
	\pinlabel $k_\alpha$ [t] at 90 8
	\pinlabel $k_\beta$ [b] at 90 172
	\pinlabel $g_i(k_\alpha)$ [tl] at 161 52
	\pinlabel $g_i(k_\beta)$ [bl] at 141 153
	
	\pinlabel $k_\alpha$ [t] at 268 8
	\pinlabel $k_\beta$ [b] at 268 172
	\pinlabel $g_i(k_\beta)$ [tl] at 340 51
	\pinlabel $g_i(k_\alpha)$ [bl] at 331 142

	\endlabellist
	\centering
	\includegraphics{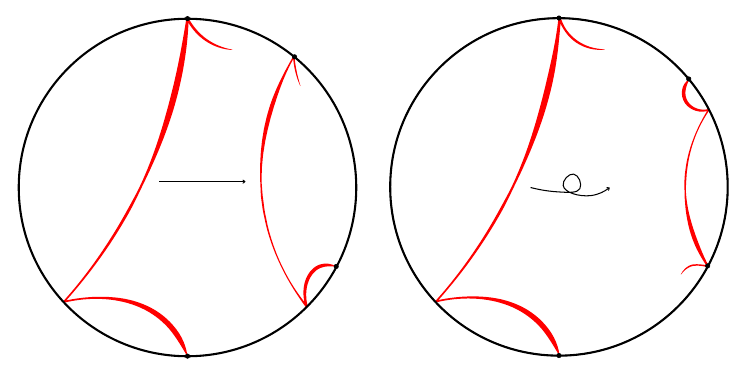}
	\caption{Translation- and rotation-like elements.} \label{figure:HomotopyClosing-TranslationRotationLike}
\end{figure}
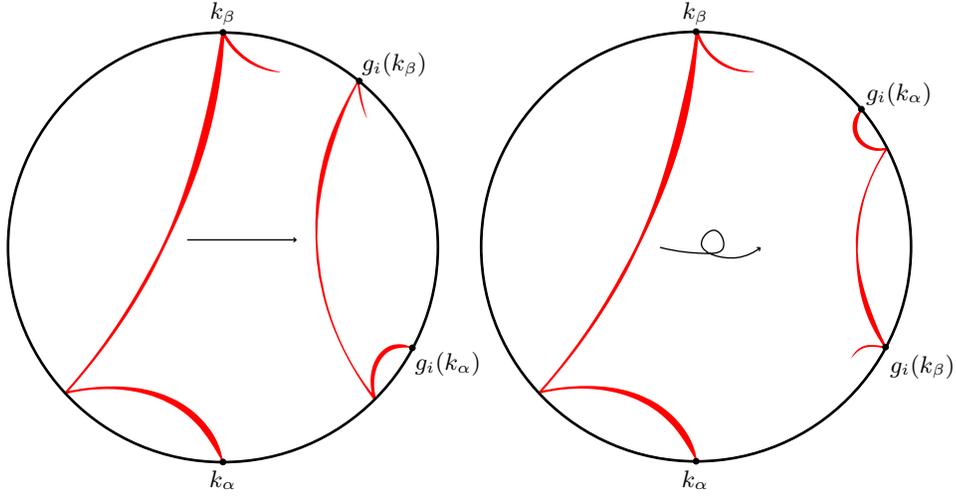

\begin{claim}\label{claim:TranslationRotationLike}
	\begin{enumerate}
		\item Let $(h_i) \subset (g_i)$ be an infinite sequence of translation-like elements. Then
		\[ \lim h_i(k_\alpha) = k_\alpha \text{ and } \lim h_i(k_\beta) = k_\beta. \]
		\item Let $(h_i) \subset (g_i)$ be an infinite sequence of rotation-like elements. Then
		\[ \lim h_i(k_\alpha) = k_\beta \text{ and } \lim h_i(k_\beta) = k_\alpha. \]
	\end{enumerate}
\end{claim}	

We will prove (1); (2) follows from a similar argument.
	
The ends of $h_i(\hK)$ are contained in $(k_\alpha, k_\beta)$, and accumulate on ends of $\hK$, so it follows that both $h_i(k_\alpha)$ and $h_i(k_\beta)$ accumulate on nontrivial subsets of $\{k_\alpha, k_\beta\}$. Thus it suffices to show that $h_i(k_\alpha)$ does not accumulate on $k_\beta$, and $h_i(k_\beta)$ does not accumulate on $k_\alpha$.
	
Suppose that $h_i(k_\alpha)$ accumulates on $k_\beta$. Since the $h_i$ are translation-like, this implies that the intervals $h_i(k_\alpha, k_\beta)$ accumulate on $k_\beta$. But this is impossible, since each $h_i(k_\alpha, k_\beta)$ contains an end of $h_i(\hL)$, which cannot accumulate on $k_\beta \in \partial \hK$. A similar argument shows that $h_i(k_\beta)$ cannot accumulate on $k_\alpha$, completing the claim.

\subsubsection{Eventual return}
Let us show that each sprig in $\sigma_{\alpha\beta}$ eventually returns to $\sigma_{\alpha\beta}$.

\begin{claim}\label{claim:EventualReturn}
	For each $\hN \in \sigma_{\alpha\beta}$, we eventually have $g_i(\hN) \in \sigma_{\alpha\beta}$.
\end{claim}

First, we note that it suffices to prove this in the special case where $\hN$ intersects the bounded part $\chK$ of $\hK$. Indeed, in the general case, we can use Lemma~\ref{lemma:DenseInSeparationInterval} to find $\hN', \hN'' \in \sigma_{\alpha\beta}$ that intersect $\chK$, with $\hN' \prec \hN \prec \hN''$. Then $g_i(\hN'), g_i(\hN'') \in \sigma_{\alpha\beta} \Rightarrow g_i(\hN) \in \sigma_{\alpha\beta}$ by Lemma~\ref{lemma:SeparationIntervalProperties}(2).
	
Let $\hN$ be an element of $\sigma_{\alpha\beta}$ that intersects $\chK$ nontrivially, and suppose for contradiction that there is an infinite subsequence $(h_i) \subset (g_i)$ such that $h_i(\hN) \notin \sigma_{\alpha\beta}$ for all $i$. Choose a point $q \in \chN \cap \chK$. Then by Coarse Contraction (Proposition~\ref{proposition:CoarseContractionSprigs}), we can pass to a subsequence for which $h_i(q)$ converges to some point $q_\infty \in \chK$. Then $\varlimsup h_i(\hN) \subset \hcD^-(q_\infty)$, and $\hcD^+(q_\infty) = \hK$. Since $q_\infty \in P$, this means means that $\varlimsup \partial h_i(\hN)$ is disjoint from $\partial \hK$.

After taking a further subsequence, we can assume that the $h_i$ are all translation-like or all rotation-like. Suppose that the $h_i$ are all translation-like. Since $\hN \in \sigma_{\alpha\beta}$, it separates $k_\alpha$ from $k_\beta$, so it must have ends $n \in (k_\alpha, k_\beta)$ and $n' \in (k_\beta, k_\alpha)$. Since $h_i$ is translation-like, $h_i(n) \in (h_i(k_\alpha), h_i(k_\beta)) \subset (k_\alpha, k_\beta)$. We must also have $h_i(n') \in (k_\alpha, k_\beta)$, since otherwise $h_i(\hN)$ would separate $k_\alpha$ from $k_\beta$. In particular, $h_i(n')$ must be contained in either $(k_\alpha, h_i(k_\alpha))$ or $(h_i(k_\beta), k_\beta)$. Then $h_i(k_\alpha) \to k_\alpha$ and $h_i(k_\beta) \to k_\beta$ (Claim~\ref{claim:TranslationRotationLike}) implies that $h_i(n')$ accumulates on either $k_\alpha$ or $k_\beta$, which contradicts the fact that $\varlimsup \partial h_i(\hN)$ is disjoint from $\partial \hK$.
	
A similar argument handles the case where the $h_i$ are rotation-like, completing the claim.

In addition, each $g_i$ preserves or reverses the order on $\sigma_{\alpha\beta}$ depending on whether it is translation- or rotation-like:

\begin{claim}
	Let $\hN, \hN' \in \sigma_{\alpha\beta}$, with $\hN \prec \hN'$, and let $j$ large enough so that $g_j(\hN), g_j(\hN') \in \sigma_{\alpha\beta}$. Then $g_j(\hN) \prec g_j(\hN')$ when $g_j$ is translation-like, and $g_j(\hN') \prec g_j(\hN)$ when $g_j$ is rotation-like.
\end{claim}

This follows easily from Lemma~\ref{lemma:SeparationIntervalEquivalencies}, using the fact that $I(\hN_\alpha, \hN_\beta) \subset (k_\alpha, k_\beta)$ and $I(\hN_\beta, \hN_\alpha) \subset (k_\beta, k_\alpha)$.

\subsubsection{Eventual contraction}
To complete the proof, we will see that the elements of $\sigma_{\alpha\beta}$ not only return to $\sigma_{\alpha\beta}$, but are pulled closer together.

Let $\Omega(p) \subset \chK$ be as in \S\ref{subsection:CoarseContraction}, and let $\Omega := \Omega(p) \cap U(\hN_\alpha, \hN_\beta)$. For each $p' \in \chK$, Proposition~\ref{proposition:CoarseContractionSprigs} says that 
\[ \varlimsup g_i(p') \subset \Omega(p). \]
In addition, if the negative sprig through $p'$ is contained in $\sigma_{\alpha\beta}$, then $p'$ is eventually contained in $U(\hN_\alpha, \hN_\beta)$ by Claim~\ref{claim:EventualReturn}, so we have 
\[ \varlimsup g_i(p') \subset \Omega. \] 

Since $\Omega$ is a compact subset of the open set $U(\hN_\alpha, \hN_\beta)$, we can find negative sprigs $\hN_a \prec \hN_b$ in $\sigma_{\alpha\beta}$ such that $\Omega \subset U(\hN_a, \hN_b)$. By Lemma~\ref{lemma:DenseInSeparationInterval}, we can assume that $\hN_a$ and $\hN_b$ intersect $\chK$, and choose points $p_a \in \chN_a \cap \chK$ and $p_b \in \chN_b \cap \chK$. See Figure~\ref{figure:HomotopyClosing-Trapping}.

\begin{figure}[h]
	\labellist
	\small\hair 2pt
	
	\pinlabel $\textcolor{blue}{\hN_\beta}$ 				[t]		at 125 	145
	\pinlabel $\textcolor{blue}{\hN_\alpha}$ 				[b] 	at 120	28
	\pinlabel $\textcolor{violet}{\Omega(p)}$  						at 127	112
	\pinlabel $\textcolor{white}{U(\hN_\alpha, \hN_\beta)}$ 		at 40 	110
	
	\pinlabel $\textcolor{blue}{\hN_b}$ 					[t] 	at 300 	135
	\pinlabel $p_b$ 										[bl] 	at 267 	145
	\pinlabel $\textcolor{blue}{\hN_a}$ 					[b] 	at 300 	58
	\pinlabel $p_a$ 										[tl] 	at 228 	57
	\pinlabel $\textcolor{white}{U(\hN_a, \hN_b)}$ at 215 100
	\pinlabel $\textcolor{violet}{\Omega}$  at 272 102
	
	\endlabellist
	\centering
	\includegraphics{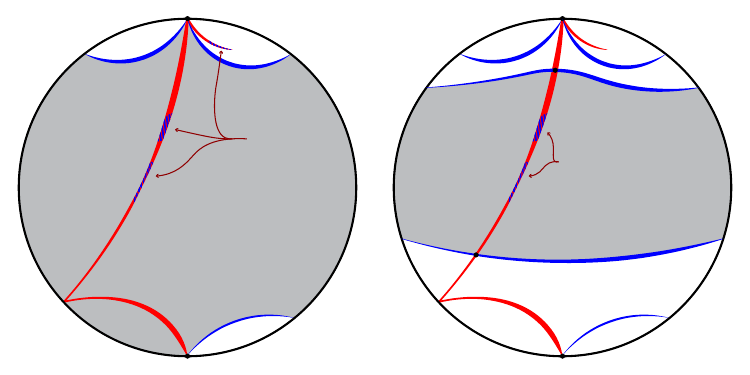}
	\caption{Eventual contraction.} \label{figure:HomotopyClosing-Trapping}
\end{figure}
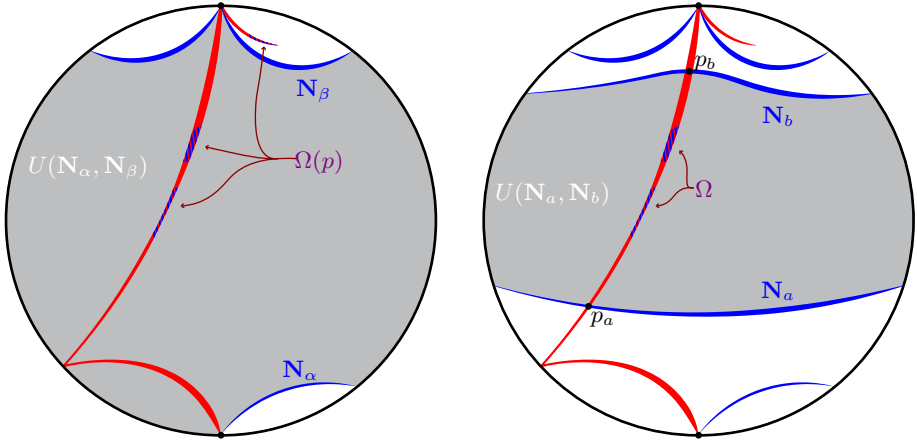

By taking $i$ sufficiently large, we can assume $g_i(\hN_a)$ and $g_i(\hN_b)$ are contained in $\sigma_{\alpha\beta}$. Furthermore, since $\varlimsup g_i(p_a)$ and $\varlimsup g_i(p_b)$ are contained in $\Omega$, which is a compact subset of the open set $U(\hN_a, \hN_b)$, we can assume that $g_i(p_a)$ and $g_i(p_b)$ are contained in $U(\hN_a, \hN_b)$. Then $g_i(\hN_a)$ and $g_i(\hN_b)$ are contained in $\sigma(\hN_a, \hN_b) \subset \sigma_{\alpha\beta}$, which means that
\[ \hN_a \prec g_i(\hN_a) \prec g_i(\hN_b) \prec \hN_b \]
when $g_i$ is translation-like, and
\[ \hN_a \prec g_i(\hN_b) \prec g_i(\hN_a) \prec \hN_b \]
when it is rotation-like. That is, $g_i$ eventually acts as an orientation-preserving or -reversing contraction on $\sigma(\hN_a, \hN_a)$, and hence fixes some sprig in this interval. Such a sprig is nontrivial because it has at least two ends (Lemma~\ref{lemma:TrivialNontrivialSprig}). This completes the proof of the Homotopy Closing Lemma.

\subsection{Closed orbits and the universal circle}\label{subsection:ClosedOrbitsUC}
Finally, we will show that certain closed orbits can be seen purely in terms of the universal circle.

An orientation-preserving group action $\Gamma \acts S^1$ is \emph{pA-like} if each nontrivial element has a power that acts with an even number of fixed points, alternately attracting and repelling.

\begin{theorem}\label{theorem:pAlike}
	The action of $\pi_1(M) \acts S^1_u$ is pA-like.
\end{theorem}
\begin{proof}
	Fix a nontrivial element $g \in \pi_1(M)$, and let $\alpha, \rho \in S^2_\infty$ be the attracting and repelling fixed points for $g \acts S^2_\infty$. Let $A$ and $R$ be the master sets rooted at these points, and note that $\partial A = \he^{-1}(\alpha)$ and $\partial R = \he^{-1}(\rho)$ are disjoint, closed, totally disconnected, and $g$-invariant.
	
	Let $F \subset S^1_u$ be the fixed point set for $g \acts S^1_u$. Then after replacing $g$ by some power, we can assume that $F \neq \emptyset$. Indeed, since $\partial A$ and $\partial R$ are closed and disjoint, they have a finite linking number $\lk(\partial A, \partial R)$ by Lemma~\ref{lemma:LinkingNumber}. Choose a complementary interval $(a, r)$ of $\partial A \cup \partial R$ with $a \in \partial A$ and $r \in \partial R$. By the proof of this lemma, there are finitely many such intervals, which we called $\partial A \partial R$-interstitial intervals, so $g^n(a, r) = (a, r)$ for some $n$. Then $a, r \in F$ as desired.
	
	Now that $F \neq \emptyset$, note that $F \subset \partial A \cup \partial R$. Indeed, if $s \in S^1_u \setminus (\partial A \cup \partial R)$, then $\he(s) \notin \{\alpha, \rho\}$, so $g(\he(s)) \neq \he(s)$ and hence $g(s) \neq s$.
	
	Let $I$ be a complementary interval of $F$. Then $g$ acts as a translation on $I$, fixing its endpoints so we can write it as either $I = (a_I, r_I)$ or $I = (r_I, a_I)$, where $a_I$ and $r_I$ are attracting and repelling with respect to points in $I$. Then we must have $a_I \in \partial A$ and $r_I \in \partial R$. Indeed, take $s \in I \setminus (\partial A \cup \partial R)$. Then $a_I = \lim_{i \to \infty} g^i(s)$, which is contained in $\partial A$ because $\lim_{i \to \infty} g_i(\he^+(s)) = \alpha$. Similarly, $r_I = \lim_{i \to -\infty} g^i(s)$, which is contained in $\partial R$ because $\lim_{i \to - \infty} g_i(\he(s)) = \rho$.
	
	Let $a \in \partial A$. If $g(a) \neq a$, then $a$ is contained in some complementary interval $J$ of $F$, and $\lim_{i \to -\infty} g^i(a) = r_J$. But $\partial A$ is $g$-invariant, so this means that $r_J \in \partial A$, a contradiction. Hence $\partial A \subset F$. Similarly, $\partial R \subset F$, so $\partial A \cup \partial R \subset F$. We already showed the opposite inclusion, we we have $F = A \cup R$.
	
	Since each complementary interval of $F$ has its attracting and repelling endpoints in $\partial A$ and $\partial R$, it suffices to show that there are only finitely many such complementary intervals. But these are $\partial A \partial R$- or $\partial R \partial A$-interstitial interval, of which there are only finitely many because $\lk(\partial A, \partial R)$ is finite.
\end{proof}

The following result characterizes the elements of $\pi_1(M)$ that do not represent multiples of closed orbits. Conversely, any nontrivial element that does \emph{not} act as in the proposition represents a multiple of the free homotopy class of some closed orbit.

\begin{proposition}\label{proposition:AttractingRepelling}
	Suppose that $g \in \pi_1(M)$ does not fix a point in $P$. Then it acts on $S^1_u$ with exactly two fixed points in an attracting-repelling pair.
\end{proposition}
\begin{proof}
	Let $F \subset S^1_u$ be the fixed point set for $g \acts S^1_u$. Then $F \neq \emptyset$. Indeed, otherwise $g$ would fix a point in $P$ by Brouwer's fixed point theorem.
	
	In the previous proof, we showed that $F = \partial A \cup \partial R$, where $A$ and $R$ are the master sets rooted at $\alpha$ and $\rho$, and the points in $\partial A$ and $\partial R$ are attracting and repelling. To complete the lemma, it suffices to show that both $\partial A$ and $\partial R$ consists of a single point. If $\partial A$ contained more than one point, then $A$ would be nontrivial, and if $\partial R$ contained more than one point then $R$ would be nontrivial. Either way, $g$ would fix a point in $P$ by Lemma~\ref{lemma:ElementPreservingLeaf}, contradicting our hypothesis.
\end{proof}

\section{Questions}\label{section:Questions}
Our proof of the Homotopy Closing Lemma holds just as well for pseudo-Anosov flows, even ones that are not quasigeodesic. In fact, it can be applied to a larger class of \emph{coarsely hyperbolic} flows, defined by the existence of a transverse pair of spidery decompositions that are coarsely contracted and expanded.

It's easy to construct coarsely hyperbolic flows that are neither quasigeodesic nor pseudo-Anosov. For example, one can start with a pseudo-Anosov flow that is not quasigeodesic and blow up some orbit to a solid cylinder or torus. These examples are quite trivial since the Homotopy Closing Lemma follows easily from the Anosov Closing Lemma for the original flow. It would be interesting to construct a less trivial class of examples.

\subsection{Product-covered flows}
A flow on a closed $3$-manifold is \emph{product covered} if the lifted flow on the universal cover is conjugate to the vertical flow on $\bR^3$. By Theorem~\ref{theorem:MontgomeryZippin}, this is equivalent to having a planar flowspace.

\begin{question}
	Let $\Phi$ be a product-covered flow on a closed hyperbolic $3$-manifold $M$. Is $\Phi$ coarsely hyperbolic?
\end{question}

More generally, we propose the following.

\begin{conjecture}
	Every product-covered flow on a closed hyperbolic $3$-manifold contains a closed orbit.
\end{conjecture}

This would be implied by the following conjecture.

\begin{conjecture}
	The fundamental group of a closed hyperbolic $3$-manifold cannot act freely and cocompactly on the plane.
\end{conjecture}

Here, a cocompact action is one whose quotient space is compact, but not necessarily Hausdorff.

\subsection{M\"obius-like groups}
In \cite{Frankel_qgflows} we proposed a very different method for proving the Closed Orbits Theorem.

An action of a group $\Gamma$ on a circle $S^1$ is called {\em M\"obius-like} if each $g \in \Gamma$ is conjugate to a M\"obius transformation. It is called {\em hyperbolic M\"obius-like} if each $g \in \Gamma$ is conjugate to a hyperbolic M\"obius transformation. A M\"obius-like or hyperbolic M\"obius-like action is called {\em M\"obius} or {\em hyperbolic M\"obius}, respectively, if it is conjugate to an action by M\"obius transformations.

The fundamental group of a closed hyperbolic 3-manifold can never act as a hyperbolic M\"obius group (see \cite{Frankel_qgflows}). The only known examples of M\"obius-like actions that are not M\"obius are found in \cite{Kovacevic}. We propose the following conjecture.

\begin{conjecture}
	The fundamental group of a closed hyperbolic 3-manifold cannot act as a hyperbolic M\"obius-like group.
\end{conjecture}

By Proposition~\ref{proposition:AttractingRepelling}, this would imply the Closed Orbits Theorem.

\end{document}